\documentclass[a4paper] {article}[12pt]

\usepackage[top=3cm, bottom=3cm, left=3.58cm, right= 3.58cm]{geometry}

\makeatletter
\renewcommand*\l@section{\@dottedtocline{1}{1.5em}{2.3em}}
\makeatother

\usepackage{amsfonts}
\usepackage{amssymb}
\usepackage[T1]{fontenc}

\usepackage{tikz}
\usetikzlibrary{calc}

\usepackage{CJK}
\usepackage{amsmath}
 
\usepackage{amsfonts}
\usepackage{amssymb}
\usepackage{amsthm}
\usepackage{amssymb}
\usepackage{enumerate}
\usepackage[calc]{picture}
\usepackage[all,cmtip]{xy}

\usepackage[mathscr]{eucal}
\usepackage{eqlist}

\usepackage{color}
\usepackage{abstract} 
\usepackage[T1]{fontenc}
 
\theoremstyle{plain}
\newtheorem{theorem}{Theorem}
\newtheorem{proposition}[theorem]{Proposition}
\newtheorem{lemma}[theorem]{Lemma}

\newtheorem{example}[theorem]{Example}
\newtheorem{corollary}[theorem]{Corollary}

\theoremstyle{definition}
\newtheorem{definition}{Definition}

\usepackage{etoolbox}
\newtheoremstyle{myrem}
 {3pt}
 {3pt}
 {\normalsize}
 { }
 {\itshape}
 {:}
 { }
 {}

 \theoremstyle{myrem}
 \newtheorem{remark}{Remark}
 \appto\remark{\leftskip\parindent}
 \appto\remark{\rightskip\parindent}

\usepackage{amsmath}
\DeclareMathOperator{\grad}{grad}

\numberwithin{equation}{section}
\numberwithin{theorem}{section}

\begin{document}

\begin{center}
{\Large{\textbf{On The Discrete Morse Functions  for   
Hypergraphs
 }}}

 \vspace{0.58cm}
 
 Shiquan Ren*, Chong Wang*,  Chengyuan Wu*, Jie Wu*

\bigskip

\bigskip
    
    \parbox{24cc}{{\small

{\textbf{Abstract}.}  
A hypergraph can be obtained from a simplicial complex by deleting some non-maximal simplices.  In this paper,  we study the embedded homology as well as the homology of the (lower-)associated simplicial  complexes for hypergraphs.  We  generalize the discrete Morse functions on simplicial complexes.  We study the discrete Morse functions on hypergraphs as well as the discrete Morse functions on the  (lower-)associated simplicial complexes  of the hypergraphs.     
}

\bigskip

}

\begin{quote}

\end{quote}

\end{center}

\vspace{1cc}

\footnotetext[1] { * first authors.  }

\section{Introduction}

Hypergraph is an important model for complex networks, for example, the collaboration network.   In mathematics,  
\begin{itemize}
\item
a hypergraph  can be   either a simplicial complex or   a  "non-closed  simplicial complex"\footnote[2]{A  simplicial complex  $\mathcal{K}$ is {\it closed}  in the following sense:
\begin{enumerate}[(a).]
\item
{\bf combinatorically closed}: for any simplex $\sigma\in \mathcal{K}$ and any non-empty subset $\tau\subseteq\sigma$,  we have $\tau\in\mathcal{K}$; 
\item
{\bf topologically closed}:  the geometric realization $|\mathcal{K}|$  (cf.  \cite{milnor})  of $\mathcal{K}$  is a closed subset of some Euclidean space $\mathbb{R}^N$.  \end{enumerate}
Note that  $(a)\Longleftrightarrow (b)$.  

\begin{quote}
\noindent {\bf Definition. }Let $\mathcal{H}$  be a hypergraph  (cf.  Definition~\ref{def-rev0}).  Let $\mathcal{K}_\mathcal{H}$  be an  (abstract)  simplicial complex such that each hyperedge   (cf.  Definition~\ref{def-rev00})  of $\mathcal{H}$  is a simplex of $\mathcal{K}_\mathcal{H}$.  Let $|\mathcal{K}_\mathcal{H}|$  be the geometric  realization of $\mathcal{K}_\mathcal{H}$  in $\mathbb{R}^N$.   For each $\sigma\in \mathcal{K}$,  let $|\sigma|$  be the geometric simplex in $|\mathcal{K}|$.  We  define the  {\it geometric realization} $|\mathcal{H}|$ of $\mathcal{H}$   to be the subset of $|\mathcal{K}_\mathcal{H}|$
 given by  
 \begin{eqnarray*}
|\mathcal{H}|= \{x\in \mathbb{R}^N\mid {\rm there~exists~} \sigma\in\mathcal{H} {\rm ~such~that~} x\in {\rm Int}(|\sigma|)\}. 
 \end{eqnarray*}
 \end{quote}

A  hypergraph $\mathcal{H}$ can be {\it non-closed}  in the following sense:  
\begin{enumerate}[(a)'.]
\item
{\bf combinatorically non-closed}: there exists an hyperedge $\sigma\in \mathcal{H}$ and an  non-empty subset $\tau\subseteq\sigma$  such that   $\tau\notin \mathcal{H}$; 
\item
{\bf topologically non-closed}:  the geometric realization $|\mathcal{H}|$  of $\mathcal{H}$  is not  a closed subset of some Euclidean space $\mathbb{R}^N$.  \end{enumerate}
Note that (a)'$\Longleftrightarrow$(b)'.  }   with some non-maximal faces missing;
\item
  a simplicial complex  can be regarded as  a  special  hypergraph with no non-maximal faces missing.  
\end{itemize}
It is natural to generalize the discrete Morse functions on simplicial complexes  (cf.  \cite{forman9} - \cite{forman3})   and consider discrete Morse functions on hypergraphs.  


\smallskip

Firstly,  We  review some backgrounds in  (a), (b) and (c).  

\smallskip

{\bf (a). Hypergraphs and Simplicial Complexes.} Let $V$ be  a totally-ordered finite set  whose elements are called {\it vertices}. Let $2^V$ denote the power-set of $V$.  Let $\emptyset$ denote the empty set.  A {\it hypergraph}  $\mathcal{H}$  on  $V$  is a subset of  $2^V\setminus\{\emptyset\}$ (cf. \cite{berge,parks}).  
  An element of $\mathcal{H}$ is called a {\it hyperedge}.  For any $\sigma\in\mathcal{H}$,    if $\sigma$ consists of $n+1$ vertices in $V$, then way say that $\sigma$  is of  dimension $n$ and write $\sigma$ as $\sigma^{(n)}$.   

An {\it (abstract) simplicial complex}  is a  hypergraph satisfying the following  condition (cf. \cite[p. 107]{hatcher}, \cite[Section~1.3]{wu1}, \cite{h1}):
 for any $\sigma\in \mathcal{H}$ and any non-empty subset $\tau\subseteq \sigma$,    we  have that   $\tau$ must be a hyperedge in  $\mathcal{H}$. 
 A hyperedge  of a simplicial complex is called a {\it simplex}.  

\smallskip

  {\bf (b). The Usual Homology of Simplicial Complexes  and  The Embedded Homology of Hypergraphs. }The homology theory for simplicial complexes  (cf.  \cite[Chapter~2]{hatcher})  is  well-known in algebraic topology.  Let $\mathcal{K}$  be a simplicial complex.  Let $R$  be a commutative ring with unit. For each $n\geq 0$,  by taking all the (formal)  $R$-linear combinations of the $n$-simplices in $\mathcal{K}$,  we have a free $R$-module $C_n(\mathcal{K};R)$.   The $n$-th  boundary map 
\begin{eqnarray*}
\partial_n:  ~~~ C_n(\mathcal{K};R)\longrightarrow C_{n-1}(\mathcal{K};R)
\end{eqnarray*}
is an $R$-linear map  such  that  
\begin{eqnarray*}
\partial_n(v_0v_1\ldots v_n)=\sum_{i=0}^n (-1)^i   v_0\ldots \widehat{v_i}\ldots v_n
\end{eqnarray*}
for any $n$-simplex $v_0v_1\ldots v_n\in \mathcal{K}$  with  $v_0\prec v_1\prec\cdots\prec v_n$.  It  can be verified  that $\partial_n\circ \partial_{n+1}=0$  for each $n\geq 0$.   The  $n$-th homology  of $\mathcal{K}$,  with coefficients in $R$,  is defined to be the quotient $R$-module 
\begin{eqnarray*}
H_n(\mathcal{K};R)= {\rm Ker}\partial_n/{\rm Im}\partial_{n+1}.  
\end{eqnarray*}

In 2019,  S. Bressan, J. Li, S. Ren and J. Wu \cite{h1} defined  an embedded homology for hypergraphs  as  an algebraic generalization of the usual homology of simplicial complexes.  Let $\mathcal{H}$  be a hypergraph.   Let $\mathcal{K}_\mathcal{H}$  be  an arbitrary simplicial complex  such that each  hyperedge  of $\mathcal{H}$  is a simplex of $\mathcal{K}_\mathcal{H}$.   For  each  $n\geq  0$,   let $\partial_n$  be  the $n$-th boundary map of $\mathcal{K}_\mathcal{H}$.    Let  $R_n(\mathcal{H})$   be   the free $R$-module consisting of all the (formal)  $R$-linear combinations of the $n$-hyperedges in $\mathcal{H}$.   The {\it infimum chain complex} of $\mathcal{H}$   is  defined   as 
\begin{eqnarray*}
\text{Inf}_n(\mathcal{H})= R_n(\mathcal{H})\cap\partial_n^{-1}(R_{n-1}(\mathcal{H})),  \text{\ \ \ } n\geq 0  
\end{eqnarray*}
and the {\it supremum chain complex} of $\mathcal{H}$   is  defined   as 
\begin{eqnarray*}
\text{Sup}_n(\mathcal{H})= R_n(\mathcal{H})+\partial_n (R_{n+1}(\mathcal{H})),  \text{\ \ \ } n\geq 0.   
\end{eqnarray*}
 It  is proved in \cite{h1}  that the homology of the infimum chain complex and the homology of the supremum chain complex are isomorphic.  This homology is  called  the {\it embedded homology} of $\mathcal{H}$.   
Particularly,   if $\mathcal{H}$  is an (abstract)  simplicial complex,  then its  embedded homology coincides with the usual homology.

\smallskip

{\bf (c).  The Discrete Morse Theories  for  Simplicial Complexes, Graphs  and Digraphs. }  
 During  the  1990s and   the  2000s, R. Forman \cite{forman9} - \cite{forman3} has developed a discrete Morse theory for simplicial complexes\footnote[3]{In fact,  in   \cite{forman9} - \cite{forman3}, R. Forman has developed a discrete Morse theory for general cell complexes.  In particular,  the discrete Morse theory  in   \cite{forman9} - \cite{forman3}  is applicable for  simplicial complexes. }.    A  discrete Morse function  $f$ on a  simplicial complex  $\mathcal{K}$  was   defined by assigning a real number $f(\sigma)$ to each simplex $\sigma\in \mathcal{K}$ such that  for any $n\geq 0$ and any $n$-simplex $\alpha^{(n)}\in\mathcal{K}$,   there exist at most one $\beta^{(n+1)}>\alpha^{(n)}$   in  $\mathcal{K}$  with $f(\beta)\leq f(\alpha)$ and  at most one $\gamma^{(n-1)}<\alpha^{(n)}$  in $\mathcal{K}$   with $f(\gamma)\geq f(\alpha)$.    The  discrete gradient vector field   $\grad   f$   on   $\mathcal{K}$  was   defined by assigning arrows from $\alpha$ to $\beta$ whenever $\alpha^{(n)}<\beta^{(n+1)}$  and $f(\alpha)\geq f(\beta)$,  for any $n\geq 0$  and any  $\alpha^{(n)},\beta^{(n+1)}\in \mathcal{K}$.  
    The critical simplices were defined as   the simplices which are neither the heads nor the tails of any arrows in $\grad   f$. 
A  chain complex consisting of the formal linear combinations of the critical simplices was constructed  and   the homology of this   chain complex is   proved to be  isomorphic to the homology of the original simplicial complex.   The  Discrete Morse theory has applications in  both pure mathematics (cf. \cite{mams, alg2, algm,  cy})  and  data technologies (cf. \cite{app1, pers,  app3,  cy}).

Inspired by   the discrete Morse theory in \cite{forman9} - \cite{forman3},  people have studied the discrete Morse theories for graphs and digraphs.  During the 2000s,  R.  Ayala, L.M.  Fern\'{a}ndez, D. Fern\'{a}ndez-Ternero,  A.  Quintero   and  J.A.  Vilches   \cite{a1} - \cite{a4}  gave a discrete Morse theory  for graphs and studied related topics.   
In  2021,  Yong Lin,  Chong Wang and  Shing-Tung Yau \cite{lin}  gave a  discrete Morse theory for digraphs.

Throughout the discrete Morse theory for simplicial complexes in \cite{forman9} - \cite{forman3},  the discrete Morse theory for graphs in \cite{a1} - \cite{a4}  and  the discrete Morse theory for digraphs  in  \cite{lin},  the homology groups,   the discrete Morse functions, the discrete gradient vector fields, and the critical simplices     play  important and fundamental roles. 

\smallskip



\smallskip

Secondly,  we introduce the motivation  of this paper.   The  paper is  motivated  by  the following  question:
\begin{quote}
\it { \bf Question. }Whether a discrete Morse theory for  hypergraphs and their embedded homology can be developed  as a generalization of the discrete Morse theory for simplicial complexes  and their   homology  (cf. \cite{forman9} - \cite{forman3})? 
\end{quote}
In this paper,   we make a  first step  towards the answer.  As a generalization of the homology theory for simplicial complexes,  we    explore  the embedded homology of hypergraphs  in Section~\ref{s2} and study   the  homomorphisms between the embedded homology of  hypergraphs  in Theorem~\ref{pr-rev-axz}.  Moreover,  as a generalization of the discrete Morse theory for simplicial complexes,  we study the    discrete Morse functions on hypergraphs  as well as the corresponding discrete gradient vector fields and critical hyperedges,  from  Section~\ref{s3}   to  Section~\ref{s6}.

\smallskip

Thirdly, we  summarize  the outline  of this paper.  
Let $\mathcal{H}$    be a hypergraph.   In  Section~\ref{s2},  we  review the definition of  the associated simplicial complex $\Delta\mathcal{H}$  (cf.  \cite{parks})  which   the smallest  simplicial complex such that each hyperedge of $\mathcal{H}$  is a simplex of $\Delta\mathcal{H}$.    We  define the lower-associated simplicial complex $\delta\mathcal{H}$  to be the largest simplicial complex such that each simplex of $\delta\mathcal{H}$  is a   hyperedge of $\mathcal{H}$.  We  prove in Theorem~\ref{pr-rev-axz}  that  a morphism of hypergraphs induces  a homomorphism of the  embedded homology,  a  homomorphism   of the homology of the  associated simplicial complexes,   and  a  homomorphism of  the homology of the lower-associated simplicial complexes.   
From  Section~\ref{s3}   to  Section~\ref{s6},   we generalize the discrete Morse functions on simplicial complexes and define the discrete Morse functions on hypergraphs.   We study the discrete gradient vector fields  of the discrete Morse functions  on hypergraphs   as well as the  corresponding critical hyperedges.  We characterize the critical hyperedges in Theorem~\ref{co-4.a}.

    \section{The  Associated Simplicial Complexes  and The Embedded Homology  for Hypergraphs }\label{s2}
    
  In this section,  we review some definitions  as well as some  basic properties about hypergraphs,  their (lower-)associated simplicial complexes,  and the  embedded homology.     We  also give  some examples and show that  the homology of the associated simplicial complexes,  the homology of the lower-associated simplicial complexes,  and the embedded homology of the hypergraphs  detect the topology of hypergraphs from  different aspects.  
  
  \smallskip

 \subsection{Hypergraphs,  The  Associated Simplicial Complexes, and  The  Lower-Associated Simplicial Complexes}

Let $V$  be a finite set with a total order $\prec$.  

\begin{definition}
A  {\it hyperedge} $\sigma$  on $V$  is a non-empty subset of $V$. 
\end{definition}

For any hyperedge $\sigma$ on $V$, we  can write $\sigma$ uniquely  as a subset 
\begin{eqnarray}\label{notation}
\sigma=\{v_0,v_1,\ldots,v_n\}
\end{eqnarray}
of $V$;  or equivalently,   we 
can also write $\sigma$  uniquely  in the form  of a sequence  
\begin{eqnarray*}
\sigma=v_0v_1\ldots v_n
\end{eqnarray*}
for some $n\geq 0$ where $v_0,v_1,\ldots, v_n\in V$  and  $v_0\prec v_1\prec\cdots \prec v_n$.  Throughout this paper,  we adopt  the former notation in (\ref{notation}) for an $n$-simplex.  
\begin{definition}
\label{def-rev00}
We say that $\sigma$ given  by  (\ref{notation})  is an {\it $n$-hyperedge}  and call $n$ the {\it dimension} of $\sigma$. 
\end{definition}

\begin{definition}[cf. \cite{berge,h1,parks}]
\label{def-rev0}
A {\it hypergraph} $\mathcal{H}$  on $V$   is a collection of hyperedges on $V$.   
\end{definition}

\begin{definition}[cf. \cite{h1,hatcher}]
An {\it (abstract)  simplicial complex} $\mathcal{K}$  on $V$  is a hypergraph  on $V$  such that for any $\sigma\in \mathcal{K}$  and any non-empty subset $\tau\subseteq \sigma$,  we always have $\tau\in\mathcal{K}$.    
\end{definition}

A  hyperedge of a simplicial complex  is also called a {\it simplex}.

\begin{definition}\label{def-rev1}
    Let $\mathcal{H}$ and $\mathcal{H}'$  be two hypergraphs.  We   say that $\mathcal{H}$  can be embedded in $\mathcal{H}'$  and write $\mathcal{H}\subseteq \mathcal{H}'$  if for any hyperedge $\sigma\in \mathcal{H}$,  we always  have that $\sigma\in \mathcal{H}'$. 
    \end{definition}

    Let $\mathcal{H}$  be a hypergraph on $V$. 
    
\begin{definition}[cf. \cite{parks}]
 \label{def-rev2}
The {\it associated simplicial complex} $\Delta \mathcal{H}$ of $\mathcal{H}$ is the smallest simplicial complex that $\mathcal{H}$ can be embedded in. 
\end{definition}

   Let  $\sigma=v_0v_1\ldots v_n$  be  an $n$-hyperedge  on $V$.  The next    lemma  is  straight-forward  from Definition~\ref{def-rev0}  and Definition~\ref{def-rev2}.

\begin{lemma}{\rm (cf. \cite[Lemma~8]{parks})}.
\label{le-rev2.1}
The {\it associated simplicial complex} $\Delta\sigma$ of $\sigma$ is   the collection of all the nonempty subsets of $\sigma$
\begin{eqnarray}\label{eq-md1}
\Delta\sigma=\{ \{v_{i_0}, v_{i_1}, \ldots, v_{i_k}\} \mid 0\leq i_0<i_1<\cdots< i_k\leq n{\rm~and~} 0\leq k\leq n\}. 
\end{eqnarray}
\qed
    \end{lemma}
    
    The  next lemma  is straight-forward from Definition~\ref{def-rev1} and Definition~\ref{def-rev2}.

    \begin{lemma}\label{le-rev2.2}
    Let  $\mathcal{H}$  and  $\mathcal{H}'$   be two hypergraphs such that $\mathcal{H}\subseteq \mathcal{H}'$.  Then $\Delta\mathcal{H}\subseteq \Delta\mathcal{H}'$.  \qed
    \end{lemma}
    
    The next lemma follows from the above. 
    
   \begin{lemma}\label{le-2.cc}
  For any hypergraph $\mathcal{H}$,  its associated simplicial complex  $\Delta\mathcal{H}$ has its set of simplices as the union of the $\Delta\sigma$'s for all $\sigma\in\mathcal{H}$,  i.e. 
\begin{eqnarray}\label{e10}
\Delta \mathcal{H}=\{\tau\in\Delta\sigma\mid \sigma\in \mathcal{H}
\}.
\end{eqnarray} 
\end{lemma}
\begin{proof}
Firstly,  the right-hand side of (\ref{e10})  is a simplicial complex whose the set of simplices contains all the hyperedges of $\mathcal{H}$.  Thus  by Definition~\ref{def-rev2},  
\begin{eqnarray}\label{e10aa}
\Delta \mathcal{H}\subseteq\{\tau\in\Delta\sigma\mid \sigma\in \mathcal{H}
\}.
\end{eqnarray}

Secondly,  we take any  hyperedge $\sigma\in\mathcal{H}$  and let $\{\sigma\}$  be the  hypergraph with the single hyperedge $\sigma$.   Since $\{\sigma\}\subseteq\mathcal{H}$,  it follows from Lemma~\ref{le-rev2.2} that $\Delta\sigma\subseteq \Delta\mathcal{H}$.  Letting $\sigma$ run over all the hyperedges in $\mathcal{H}$,  it follows that 
\begin{eqnarray}\label{e10bb}
\{\tau\in\Delta\sigma\mid \sigma\in \mathcal{H}\}&=&\cup_{\sigma\in\mathcal{H}} \Delta\sigma \nonumber
\\&\subseteq&\Delta \mathcal{H}.
\end{eqnarray}  

Summarizing (\ref{e10aa})  and (\ref{e10bb}), we  obtain (\ref{e10}).  
\end{proof}

\begin{definition}\label{def-rev2.3}
The {\it lower-associated simplicial complex} $\delta \mathcal{H}$ of $\mathcal{H}$  is  the largest simplicial complex that can be embedded in $\mathcal{H}$.
\end{definition}

The next lemma follows from Definition~\ref{def-rev2.3}.  

\begin{lemma}\label{le-2.dd}
 Then the set of simplices of $\delta\mathcal{H}$ consists of the hyperedges  $\sigma\in\mathcal{H}$ whose associated simplicial complexes $\Delta\sigma$ are subsets of $\mathcal{H}$.  In other words,  
\begin{eqnarray}\label{e11}
\delta\mathcal{H}&=&\{\sigma\in\mathcal{H}\mid \Delta\sigma\subseteq \mathcal{H}\}\nonumber\\
&=&\{\tau\in\Delta\sigma\mid \Delta\sigma\subseteq \mathcal{H}\}. 
\end{eqnarray}
\end{lemma}
\begin{proof}
Firstly,   for any $\sigma \in \mathcal{H}$,   if $\Delta\sigma\subseteq\mathcal{H}$,  then by Definition~\ref{def-rev2.3},  we  have 
\begin{eqnarray*}
\Delta\sigma\subseteq \delta\mathcal{H}. 
\end{eqnarray*}  
Letting $\sigma$ run over all the hyperedges of $\mathcal{H}$,  it follows that 
\begin{eqnarray*}
\{\tau\in\Delta\sigma\mid \Delta\sigma\subseteq \mathcal{H}\}&=& \cup_{\Delta\sigma\subseteq\mathcal{H}} \Delta\sigma \\
&\subseteq& \delta\mathcal{H}. 
\end{eqnarray*}

Secondly,  let $\sigma$  be any simplex of $\delta\mathcal{H}$.  Then by Definition~\ref{def-rev2.3},  we have $\Delta\sigma\subseteq \mathcal{H}$.  Letting $\sigma$ run over all the hyperedges of $\mathcal{H}$,  it follows that 
\begin{eqnarray*}
 \delta\mathcal{H}\subseteq\{\sigma\in\mathcal{H}\mid \Delta\sigma\subseteq \mathcal{H}\}. 
\end{eqnarray*}

Thirdly,  since $\sigma\in\Delta\sigma$,  we have
\begin{eqnarray*}
\{\sigma\in\mathcal{H}\mid \Delta\sigma\subseteq \mathcal{H}\} 
\subseteq \{\tau\in\Delta\sigma\mid \Delta\sigma\subseteq \mathcal{H}\}. 
\end{eqnarray*}

Summarizing all the three points,  we obtain (\ref{e11}).  
\end{proof}

The next proposition follows  from Lemma~\ref{le-2.cc}  and Lemma~\ref{le-2.dd}.  

\begin{proposition}\label{pr-2.ee}
As hypergraphs,  
\begin{eqnarray}\label{e1}
\delta \mathcal{H}\subseteq \mathcal{H}\subseteq\Delta\mathcal{H}.
\end{eqnarray}
Moreover,  one of the equalities holds iff. both of the equalities hold iff. $\mathcal{H}$ is a simplicial complex.   
\end{proposition}

\begin{proof}
The relations (\ref{e1}) follow  from Definition~\ref{def-rev2} and Definition~\ref{def-rev2.3}  directly. 
 Alternatively,  (\ref{e1}) also  follow  from Lemma~\ref{le-2.cc}  and Lemma~\ref{le-2.dd}. 
  In the following,  we will  prove the second assertion.  We divide the proof into two steps. 

{\sc Step~1}.  $\delta \mathcal{H}= \mathcal{H}$ $\Longleftrightarrow$  $\mathcal{H}$ is a simplicial complex.   

($\Longleftarrow$):  Obvious. 

($\Longrightarrow$): Suppose $\delta\mathcal{H}=\mathcal{H}$.  Then by Lemma~\ref{le-2.dd},   it follows  that for any $\sigma\in\mathcal{H}$,   we always have $\Delta\sigma\subseteq \mathcal{H}$.   Hence $\mathcal{H}$  is a simplicial complex.

{\sc Step~2}.  $\Delta \mathcal{H}= \mathcal{H}$ $\Longleftrightarrow$  $\mathcal{H}$ is a simplicial complex.

($\Longleftarrow$): Obvious.

($\Longrightarrow$): Suppose $\Delta\mathcal{H}=\mathcal{H}$.  Then by Lemma~\ref{le-2.cc},   it follows  that for any $\sigma\in\mathcal{H}$ and any non-empty subset $\tau\subseteq\sigma$,  we always have $\tau\in \mathcal{H}$.   Hence $\mathcal{H}$  is a simplicial complex.  

Summarizing both Step~1 and Step~2,  we finish the proof.  
\end{proof}

\smallskip

\subsection{Chain Complexes  and Homology Groups  for Hypergraphs}

Let $R$  be a commutative ring with unit.   Recall that any simplicial complex $\mathcal{K}$  gives a chain complex $C_*(\mathcal{K};R)$.  Let  $\mathcal{H}$  be   any hypergraph  on $V$.   The associated simplicial complex $\Delta\mathcal{H}$  gives a chain complex $C_*(\Delta\mathcal{H};R)$.  We use 
\begin{eqnarray*}
\partial_n:~~~ C_n(\Delta\mathcal{H};R)\longrightarrow C_{n-1}(\Delta\mathcal{H};R), \text{\ \ } n=0,1,2,\ldots,
\end{eqnarray*}
 to denote the boundary maps of this chain complex.  For simplicity, we sometimes denote $\partial_n$ as $\partial$ and omit the dimension $n$.

By Definition~\ref{def-rev2},  Definition~\ref{def-rev2.3}  and Proposition~\ref{pr-2.ee},  the lower-associated simplicial complex $\delta\mathcal{H}$  is a simplicial sub-complex of $\Delta\mathcal{H}$.   Consequently, $\delta\mathcal{H}$  gives a sub-chain complex $C_*(\delta\mathcal{H};R)$  of $C_*(\Delta\mathcal{H};R)$.  
The boundary map of $C_*(\delta\mathcal{H};R)$  is the restriction of $\partial_*$ to $C_*(\delta\mathcal{H};R)$,  i.e. 
\begin{eqnarray*}
\partial_n\mid_{C_n(\delta\mathcal{H};R)}: ~~~C_n(\delta\mathcal{H};R)\longrightarrow C_{n-1}(\delta\mathcal{H};R), \text{\ \ }n=0,1,2,\ldots. 
\end{eqnarray*}

Let $D_*$ be a graded sub-$R$-module of $C_*(\Delta\mathcal{H};R)$. 

\begin{definition}{\rm (cf. \cite[Section~2]{h1})}.
The {\it infimum chain complex} $\{\text{Inf}_n(D_*,C_*(\Delta\mathcal{H};R))\}_{n\geq 0}$    is  the largest sub-chain complex of $C_*(\Delta\mathcal{H};R)$ contained in $D_*$ as graded $R$-modules.  
\end{definition}

\begin{lemma} {\rm (cf. \cite[Section~2]{h1})}.
The infimum chain complex  can be expressed explicitly as 
\begin{eqnarray*}
{\rm  Inf}_n(D_*,C_*(\Delta\mathcal{H};R))= D_n\cap\partial_n^{-1}(D_{n-1}),  \text{\ \ \ } n\geq 0.  
\end{eqnarray*}
\qed 
\end{lemma}

\begin{definition}{\rm (cf. \cite[Section~2]{h1})}.
The {\it supremum chain complex} $\{\text{Sup}_n(D_*,C_*(\Delta\mathcal{H};R))\}_{n\geq 0}$   is the smallest sub-chain complex of $C_*(\Delta\mathcal{H};R)$ containing $D_*$ as graded $R$-modules.  \end{definition}

\begin{lemma}{\rm (cf.  \cite[Section~2]{h1})}.
  The supremum chain complex   can be expressed explicitly as 
\begin{eqnarray*}
{\rm Sup}_n(D_*,C_*(\Delta\mathcal{H};R))= D_n+\partial_{n+1}(D_{n+1}), \text{\ \ \ }n\geq 0,
\end{eqnarray*}
\qed
\end{lemma}

It is obvious that 
 as chain complexes,    we have  
\begin{eqnarray}
&&\{\text{Inf}_n(D_*,C_*(\Delta\mathcal{H};R)), \partial_n\mid_{\text{Inf}_n(D_*,C_*(\Delta\mathcal{H};R))}\}_{n\geq 0}\nonumber\\
&\subseteq& \{\text{Sup}_n(D_*,C_*(\Delta\mathcal{H};R)), \partial_n\mid_{\text{Sup}_n(D_*,C_*(\Delta\mathcal{H};R))}\}_{n\geq 0} \nonumber\\
&\subseteq& \{C_n(\Delta\mathcal{H};R), \partial_n\}_{n\geq 0}. 
\label{eq-2.99}
\end{eqnarray}
Consider  the canonical inclusion 
\begin{eqnarray*}
\iota: ~~~\text{Inf}_n(D_*,C_*(\Delta\mathcal{H};R))\longrightarrow \text{Sup}_n(D_*,C_*(\Delta\mathcal{H};R)), \text{\ \ \ }n\geq 0.  
\end{eqnarray*}
We  have the next lemma. 
\begin{lemma}\label{le-rev8}
The canonical inclusion  $\iota$ induces an isomorphism of homology groups
\begin{eqnarray}\label{eq-0.1}
\iota_*: &&H_*(\{{\rm Inf}_n(D_*,C_*(\Delta\mathcal{H};R)), \partial_n\mid_{{\rm Inf}_n(D_*)}\}_{n\geq 0})\nonumber\\
&&\overset{\cong}{\longrightarrow}H_*(\{{\rm Sup}_n(D_*,C_*(\Delta\mathcal{H};R)), \partial_n\mid_{{\rm Sup}_n(D_*)}\}_{n\geq 0}). 
\end{eqnarray}
\end{lemma}

\begin{proof}
Let $k\geq 0$  be an integer.  By the definition of homology groups, we have 
\begin{eqnarray*}
&& H_k(\text{Sup}_*(D_*,C_*(\Delta\mathcal{H};R)))\\
 &=&\text{Ker}(\partial_k|_{D_k+\partial_{k+1}D_{k+1}})/\text{Im}(\partial_{k+1}|_{D_{k+1}+\partial_{k+2} D_{k+2}})  
 \end{eqnarray*}
 and 
 \begin{eqnarray*}
 &&H_k(\text{Inf}_*(D_*,C_*(\Delta\mathcal{H};R)))\\
 &=&\text{Ker}(\partial_k |_{D_k\cap\partial_k^{-1}D_{k-1}}) /\text{Im} (\partial_{k+1} |_{D_{k+1}\cap\partial_{k+1}^{-1}D_{k}})
 \end{eqnarray*}
Moreover,   by the proof of \cite[Proposition~2.4]{h1} or by  a direct calculation,  
we  have 
 \begin{eqnarray*}
 \text{Ker}(\partial_k|_{D_k+\partial_{k+1}D_{k+1}})&=& \partial_{k+1}D_{k+1}+\text{Ker}(\partial_k|_{D_k}),\\
 \text{Im}(\partial_{k+1}|_{D_{k+1}+\partial_{k+2} D_{k+2}})&=&\partial_{k+1} D_{k+1}, \\
\text{Ker}(\partial_k |_{D_k\cap\partial_k^{-1}D_{k-1}})&=&\text{Ker}(\partial_k|_{D_k}),\\
\text{Im} (\partial_{k+1} |_{D_{k+1}\cap\partial_{k+1}^{-1}D_{k}}) &=&\text{Ker}(\partial_n|_{D_k})\cap \partial_{k+1} D_{k+1}. 
\end{eqnarray*}
For each $k\geq 0$,  the canonical inclusion $\iota$  induces a homomorphism 
\begin{eqnarray*} 
\iota_*:  &&H_k(\{\text{Inf}_n(D_*,C_*(\Delta\mathcal{H};R)), \partial_n\mid_{\text{Inf}_n(D_*)}\}_{n\geq 0})\nonumber\\
&&  \longrightarrow H_k(\{\text{Sup}_n(D_*,C_*(\Delta\mathcal{H};R)), \partial_n\mid_{\text{Sup}_n(D_*)}\}_{n\geq 0}) 
\end{eqnarray*}
 of homology groups.  Precisely,  $\iota_*$  sends an element 
 \begin{eqnarray*}
d_k +\text{Ker}(\partial_n|_{D_k})\cap \partial_{k+1} D_{k+1}, 
 \end{eqnarray*} 
 where  $d_k\in D_k$  such that  $\partial_k d_k=0$,  
 in  the quotient $R$-module 
 \begin{eqnarray*}
 \text{Ker}(\partial_k|_{D_k})/\big(\text{Ker}(\partial_n|_{D_k})\cap \partial_{k+1} D_{k+1}\big)
 \end{eqnarray*}
 to the element 
 \begin{eqnarray*}
 d_k +\partial_{k+1} D_{k+1}   
 \end{eqnarray*} 
in  the quotient $R$-module 
 \begin{eqnarray*}
 \big(\partial_{k+1}D_{k+1}+\text{Ker}(\partial_k|_{D_k})\big)/\partial_{k+1} D_{k+1}.  
 \end{eqnarray*}
 By the isomorphism theorem of modules,  we  have that $\iota_*$  is an isomorphism.  
\end{proof}

\begin{remark}
In \cite[Proposition~2.4]{h1},  it is proved that for each $k\geq 0$,  the homology groups  
\begin{eqnarray*}
H_k(\{\text{Inf}_n(D_*,C_*(\Delta\mathcal{H};R)), \partial_n\mid_{\text{Inf}_n(D_*)}\}_{n\geq 0})
\end{eqnarray*}
  and 
\begin{eqnarray*}
H_*(\{\text{Sup}_n(D_*,C_*(\Delta\mathcal{H};R)), \partial_n\mid_{\text{Sup}_n(D_*)}\}_{n\geq 0})
\end{eqnarray*}
  are isomorphic.  Here in Lemma\ref{le-rev8},   we strenthen  \cite[Proposition~2.4]{h1}  and prove that the canonical inclusion induces such an isomorphism.  
\end{remark}

With the help of Lemma~\ref{le-rev8},  the {\it embedded homology groups} of $D_*$  can be  defined:  

\begin{definition} {\rm (cf. \cite[Section~2]{h1})}.   
\label{def-rev321}
We call the homology groups in (\ref{eq-0.1}) the {\it embedded homology} of $D_*$   and denote the homology groups as $H_n(D_*,C_*(\Delta\mathcal{H};R))$, $n\geq 0$.  
\end{definition}

For each $n\geq 0$,  we use  $R(\mathcal{H})_n$  to denote the free $R$-module generated by all the $n$-hyperedges of $\mathcal{H}$.  We consider the specific graded  sub-$R$-module  
\begin{eqnarray*}
D_n= R(\mathcal{H})_n, \text{\ \ \ }n\geq 0  
\end{eqnarray*} 
of $C_*(\Delta\mathcal{H};R)$.   As  a special case  of  Definition~\ref{def-rev321},  the {\it embedded homology groups} of $\mathcal{H}$  can be  defined:
\begin{definition}
  {\rm (cf. \cite[Subsection~3.2]{h1})}.
  We call the homology groups 
  \begin{eqnarray*}
H_n(\mathcal{H};R)= H_n(R(\mathcal{H})_*,C_*(\Delta\mathcal{H};R)), \text{\ \ \ }n\geq 0 
\end{eqnarray*}
 the {\it embedded homology} of $\mathcal{H}$ with coefficients in $R$  and denote the homology groups as $H_n(\mathcal{H};R)$, $n\geq 0$.  
 \end{definition}
 
 \smallskip

 \subsection{Homomorphisms Among The Homology Groups for Hypergraphs}

 We  have the next lemma. 
 
 \begin{lemma}\label{le-rev-aza}
 The canonical inclusions 
 \begin{eqnarray*}
 \delta\mathcal{H}\overset{\iota^\delta}{\longrightarrow} \mathcal{H}\overset{\iota^\Delta}{\longrightarrow}\Delta\mathcal{H}
 \end{eqnarray*}
 induce homomorphisms 
  \begin{eqnarray*}
 H_*(\delta\mathcal{H};R)\overset{(\iota^\delta)_*}{\longrightarrow} H_*(\mathcal{H};R)\overset{(\iota^\Delta)_*}{\longrightarrow}H_*(\Delta\mathcal{H};R)
 \end{eqnarray*}
 of homology groups.  
 \end{lemma}
 \begin{proof}
 As chain complexes,     we  have the inclusions 
\begin{eqnarray*}
&~&\{C_n(\delta\mathcal{H};R), \partial_n\mid_{C_n(\delta\mathcal{H};R)}\}_{n\geq 0}\\
&\subseteq&\{\text{Inf}_n(R(\mathcal{H})_*), \partial_n\mid_{\text{Inf}_n(R(\mathcal{H})_*)}\}_{n\geq 0}\\
&\subseteq& \{\text{Sup}_n(R(\mathcal{H})_*), \partial_n\mid_{\text{Sup}_n(R(\mathcal{H})_*)}\}_{n\geq 0}  \\
&\subseteq& \{C_n(\Delta\mathcal{H};R), \partial_n\}_{n\geq 0}. 
\end{eqnarray*}
The inclusion $\iota^\delta$  of hypergraphs induces an inclusion 
\begin{eqnarray*}
(\iota^\delta)_\#:~~~   C_n(\delta\mathcal{H};R)  \longrightarrow  \text{Inf}_n(R(\mathcal{H})_*),~~~~~~n\geq 0 
\end{eqnarray*}
of chain complexes,  which induces a homomorphism $(\iota^\delta)_*$   of  the  homology groups.  The inclusion $\iota^\Delta$  of hypergraphs induces an inclusion 
\begin{eqnarray*}
(\iota^\Delta)_\#:~~~    \text{Sup}_n(R(\mathcal{H})_*)  \longrightarrow C_n(\Delta\mathcal{H};R)   ~~~~~~n\geq 0 
\end{eqnarray*}
of chain complexes,  which induces a homomorphism $(\iota^\Delta)_*$   of the  homology groups.   
 \end{proof}

Let $V$  and $V'$ be two totally-ordered finite sets.  Let $\mathcal{H}$ be  a hypergraph on $V$  and $\mathcal{H}'$ be  a  hypergraph  on $V'$. 
\begin{definition}{\rm (cf. \cite[Subsection~3.1]{h1})}.
A {\it morphism} of hypergraphs from $\mathcal{H}$ to $\mathcal{H}'$ is  a map $\varphi: V\longrightarrow V' $ (i.e. $\varphi$  is a  map sending a vertex of $\mathcal{H}$ to a vertex of $\mathcal{H}'$) such that whenever $\sigma=\{ v_0, v_1,\ldots, v_k \}$ is a $k$-hyperedge of $\mathcal{H}$,   we  always have that 
\begin{eqnarray}\label{eq-zxs}
\varphi(\sigma) =  \{\varphi(v_0), \varphi(v_1), \ldots,  \varphi(v_k)\}
\end{eqnarray}
  is an $l$-hyperedge of $\mathcal{H}'$  for some $0\leq l\leq k$,  where $l$  is the number of distinct vertices among   $\varphi(v_0)$, $\varphi(v_1)$, $\ldots$,  $\varphi(v_k)$.  
\end{definition}

Let $\varphi:\mathcal{H}\longrightarrow\mathcal{H}'$ be  a morphism  of hypergraphs. 
 By an  argument similar to \cite[Section~3.1]{h1},  we  have the next lemma.  

\begin{lemma}\label{le-bvc}
A   morphism 
$\varphi:\mathcal{H}\longrightarrow\mathcal{H}'$  of hypergraphs  induces two simplicial maps 
\begin{eqnarray*}
\delta\varphi: ~~~ \delta \mathcal{H}\longrightarrow \delta\mathcal{H}' 
\end{eqnarray*}
and 
\begin{eqnarray*}
\Delta\varphi:  ~~~\Delta \mathcal{H}\longrightarrow \Delta\mathcal{H}'
\end{eqnarray*}
such that $\varphi=(\Delta\varphi)\mid_{\mathcal{H}}$ and $\delta\varphi=\varphi\mid_{\delta\mathcal{H}}$.  
\end{lemma}

\begin{proof}
Let $n\geq 0$.  Let $\{v_0,v_1,\ldots, v_n\}$ (these $n+1$ vertices are distinct) be an $n$-hyperedge of $\mathcal{H}$.  Since $\varphi$  is a morphism of hypergraphs,  we have that $\{\varphi(v_0),\varphi(v_1),\ldots, \varphi(v_n)\}$ (these $n+1$  vertices may not be distinct) is an $m$-hyperedge of $\mathcal{H}'$  for some $0\leq m\leq n$.

  For any $k\geq 0$ and any $k$-simplex  $\{u_0,u_1,\ldots, u_k\}$  of $\delta\mathcal{H}$,  we define  
  \begin{eqnarray}\label{eqa123}
  (\delta\varphi) (\{u_0,u_1,\ldots, u_k\})=\{\varphi(u_0),\varphi(u_1),\ldots, \varphi(u_k)\}.  
  \end{eqnarray}
By Definition~\ref{def-rev2.3},  there exists some $n\geq k$  and some $n$-hyperedge $\{v_0,v_1,\ldots, v_n\}$  of $\mathcal{H}$  such that $\{u_0,u_1,\ldots, u_k\}$  is a subset of $\{v_0,v_1,\ldots, v_n\}$.  Since  $\{\varphi(v_0),\varphi(v_1),\ldots, \varphi(v_n) \}$    is an  $m$-hyperedge  of $\mathcal{H}'$  and $\{\varphi(u_0),\varphi(u_1),\ldots, \varphi(u_k)\}$  is a subset of  $\{\varphi(v_0),\varphi(v_1),\ldots, \varphi(v_n)\} $,    it follows that $\{\varphi(u_0),\varphi(u_1),\ldots, \varphi(u_k)\} $  is an $l$-simplex of $\delta\mathcal{H}'$  for some $0\leq l\leq k$.   Therefore, $\delta\varphi$  is a simplicial map.

For any $k\geq 0$ and any $k$-simplex  $\{w_0,w_1,\ldots, w_k\}$  of $\Delta\mathcal{H}$,  we define  
  \begin{eqnarray}\label{eqb123}
  (\Delta\varphi) (\{w_0,w_1,\ldots, w_k\})=\{\varphi(w_0),\varphi(w_1),\ldots, \varphi(w_k)\}.  
  \end{eqnarray}
By Definition~\ref{def-rev2},   for any $0\leq l\leq k$  and any subset $\{v_0,v_1,\ldots, v_l\}$  of  $\{w_0,w_1,\ldots, w_k\}$,  we  always have that   $\{v_0,v_1,\ldots, v_l\}$   is an $l$-hyperedge of $\mathcal{H}$.  Hence  there exists some  $0\leq m\leq l$  such that  $\{\varphi(v_0),\varphi(v_1),\ldots, \varphi(v_l)\}$  is an $m$-hyperedge of $\mathcal{H}'$.   This implies that for any subset  $\{\varphi(v_0),\varphi(v_1),\ldots, \varphi(v_l)\}$  of  $\{\varphi(w_0),\varphi(w_1),\ldots, \varphi(w_k)\}$,   we always have that $\{\varphi(v_0),\varphi(v_1),\ldots, \varphi(v_l)\}$  is an $m$-hyperedge of $\mathcal{H}'$  for some $0\leq m\leq l$.  Consequently,  $\{\varphi(w_0),\varphi(w_1),\ldots, \varphi(w_k)\}$  is a simplex of  $\Delta\mathcal{H}'$.  Therefore,  $\Delta\varphi$  is a simplicial map.

Finally,  it is obvious that $\varphi=(\Delta\varphi)\mid_{\mathcal{H}}$ and $\delta\varphi=\varphi\mid_{\delta\mathcal{H}}$.  
\end{proof}

The next theorem follows from Proposition~\ref{pr-2.ee}, Lemma~\ref{le-rev-aza}, Lemma~\ref{le-bvc} and \cite[Proposition~3.7]{h1}.  

\begin{theorem}\label{pr-rev-axz}
A   morphism 
$\varphi:\mathcal{H}\longrightarrow\mathcal{H}'$  of hypergraphs  induces  homomorphisms between the homology groups 
\begin{eqnarray}
(\delta\varphi)_*:&& H_*(\delta\mathcal{H})\longrightarrow H_*(\delta\mathcal{H}'),\label{e88-1}\\
(\Delta\varphi)_*:&& H_*(\Delta\mathcal{H})\longrightarrow H_*(\Delta\mathcal{H}'),  \label{e88-2}\\
\varphi_*:&& H_*(\mathcal{H})\longrightarrow H_*(\mathcal{H}') 
\label{e88-3} 
\end{eqnarray}
such that the following diagram commutes
\begin{eqnarray*}
\xymatrix{
H_*(\delta\mathcal{H};R)\ar[r]^{(\iota^\delta)_*}\ar[d]_{(\delta\varphi)_*} & H_*(\mathcal{H};R)\ar[r]^{(\iota^\Delta)_*} \ar[d]_{\varphi_*}  &H_*(\Delta\mathcal{H};R)\ar[d]_{(\Delta\varphi)_*}   \\
H_*(\delta\mathcal{H}';R)\ar[r]^{(\iota'^\delta)_*}  & H_*(\mathcal{H}';R)\ar[r]^{(\iota'^\Delta)_*}  &H_*(\Delta\mathcal{H}';R).  
}
\end{eqnarray*}
In addition, if both $\mathcal{H}$ and $\mathcal{H}'$ are simplicial complexes, then $\varphi$ is a simplicial map     and  the three homomorphisms $(\delta\varphi)_*$, $(\Delta\varphi)_*$ and $\varphi_*$  in (\ref{e88-1}), (\ref{e88-2})  and (\ref{e88-3})  are the same. 
\end{theorem}
\begin{proof}
The simplicial maps $\delta\varphi$ and $\Delta\varphi$   induce  homomorphisms between the homology groups 
(\ref{e88-1})  and (\ref{e88-2})  respectively.  
Moreover, by \cite[Proposition~3.7]{h1}, we have an induced homomorphism between the embedded homology 
(\ref{e88-3}).  
By a direct diagram chasing,  we can prove that the diagram  in Theorem~\ref{pr-rev-axz}commutes.

Suppose in addition that both $\mathcal{H}$ and $\mathcal{H}'$ are simplicial complexes.  Then by Proposition~\ref{pr-2.ee},  we  have $\Delta\mathcal{H}=\delta\mathcal{H}=\mathcal{H}$  and $\Delta\mathcal{H}'=\delta\mathcal{H}'=\mathcal{H}'$.    It follows from  the  commutative diagram that  all the homomorphisms $(\iota^\delta)_*$,  $(\iota^\Delta)_*$,  $(\iota'^\delta)_*$ and $(\iota'^\Delta)_*$  are isomorphisms.   
It also follows from  
(\ref{eq-zxs}), (\ref{eqa123}) and (\ref{eqb123}) that  the three homomorphisms $(\delta\varphi)_*$, $(\Delta\varphi)_*$ and $\varphi_*$  in (\ref{e88-1}), (\ref{e88-2})  and (\ref{e88-3})  are the same. 
\end{proof}

\smallskip

\subsection{Examples}

The associated simplicial complex,  the lower-associated simplicial complex and the embedded homology detect the topology of hypergraphs from different aspects.   
 Given two hypergraphs $\mathcal{H}$ and $\mathcal{H}'$, we consider the following  conditions:
 \begin{enumerate}[(1).]
\item
 $\delta \mathcal{H}= \delta \mathcal{H}'$;
\item
  $\Delta \mathcal{H}= \Delta \mathcal{H}'$;
\item
  $H_*( \mathcal{H})\cong H_*(\mathcal{H}')$.
  \end{enumerate}

  \smallskip

 The next example shows that (1) and (2) cannot imply (3). 

\begin{example}\label{ex82}
Let 
\begin{eqnarray*}
\mathcal{H}&=&\{ \{v_0,v_1,v_2,v_3\},\{v_0\}\},\\
\mathcal{H}'&=&\{ \{v_0,v_1,v_2,v_3\}, \{v_0,v_1\}, \{v_0,v_2\},\{v_0,v_3\},\{v_1,v_2\},\{v_1,v_3\},\{v_2,v_3\},\{v_0\}\}. 
\end{eqnarray*}
  Then  $\delta\mathcal{H}=\delta\mathcal{H}'=\{\{v_0\}\}$.  Moreover,  both $\Delta\mathcal{H}$ and $\Delta\mathcal{H}'$ are the tetrahedron.   Furthermore,  
$H_1(\mathcal{H})=0$ and $H_1(\mathcal{H}')=\mathbb{Z}^{\oplus 3}$. 
\end{example}

\noindent The next example shows that (1) and (3) cannot imply (2).  
 
\begin{example}(see Figure~1.)\label{ex81}
Let 
\begin{eqnarray*}
\mathcal{H}&=&\{\{v_0\},\{v_1\},\{v_2\}, \{v_3\},\{v_4\},\{v_5\},\\
&&\{v_0,v_1,v_3\}, \{v_1,v_2,v_4\},\{v_3,v_4,v_5\}\},\\
\mathcal{H}'&=&\{\{v_0\},\{v_1\},\{v_2\}, \{v_3\},\{v_4\},\{v_5\},\\
&&\{v_0,v_1,v_3\}, \{v_1,v_2,v_4\},\{v_1,v_3,v_4\},\{v_3,v_4,v_5\}\}. 
\end{eqnarray*}
  Then  $\delta\mathcal{H}=\delta\mathcal{H}'=\{\{v_0\},\{v_1\},\{v_2\}, \{v_3\},\{v_4\},\{v_5\}\}$.   Moreover,  
$H_n(\mathcal{H})=H_n(\mathcal{H}')=0$ for $n\geq 1$  and $H_0(\mathcal{H})=H_0(\mathcal{H}')=\mathbb{Z}^{\oplus 6}$.  
Furthermore, $H_1(\Delta\mathcal{H})=\mathbb{Z}$, and $H_1(\Delta{\mathcal{H}'})=0$.

\begin{figure}[!htbp]
 \begin{center}
\begin{tikzpicture}
\coordinate [label=below left:$v_0$]    (A) at (1,0); 
 \coordinate [label=below right:$v_1$]   (B) at (2.5,0); 
 \coordinate  [label=below right:$v_2$]   (C) at (4,0); 
\coordinate  [label=left:$v_3$]   (D) at (1.75,2/2); 
\coordinate  [label=right:$v_4$]   (E) at (6.5/2,2/2); 
\coordinate  [label=right:$v_5$]   (F) at (5/2,4/2);

\fill (1,0) circle (2.5pt);
\fill (2.5,0) circle (2.5pt);
\fill (4,0) circle (2.5pt);
\fill (1.75,2/2) circle (2.5pt); 
\fill (6.5/2,2/2) circle (2.5pt); 
\fill (5/2,4/2) circle (2.5pt);

 \coordinate[label=left:$\mathcal{H}$:] (G) at (0.5/2,2/2);
 \draw [dashed,thick] (A) -- (B);
 \draw [dashed,thick] (B) -- (C);
  \draw [dashed,thick] (D) -- (A);
\draw [dashed,thick] (D) -- (B);
\draw [dashed,thick] (E) -- (C);
\draw [dashed,thick] (E) -- (B);
\draw [dashed,thick] (E) -- (F);
\draw [dashed,thick] (D) -- (F);
\draw [dashed, thick] (D) -- (E);

\fill [fill opacity=0.25][gray!100!white] (A) -- (B) -- (D) -- cycle;
\fill [fill opacity=0.25][gray!100!white] (B) -- (C) -- (E) -- cycle;
\fill [fill opacity=0.25][gray!100!white] (D) -- (E) -- (F) -- cycle;

\coordinate [label=below left:$v_0$]    (A) at (1+6,0); 
 \coordinate [label=below right:$v_1$]   (B) at (2.5+6,0); 
 \coordinate  [label=below right:$v_2$]   (C) at (4+6,0); 
\coordinate  [label=left:$v_3$]   (D) at (1.75+6,2/2); 
\coordinate  [label=right:$v_4$]   (E) at (6.5/2+6,2/2); 
\coordinate  [label=right:$v_5$]   (F) at (5/2+6,4/2); 

 \coordinate[label=left:$\Delta{\mathcal{H}}$:] (G) at (0.5/2+6,2/2);
 \draw  [thick](A) -- (B);
 \draw [thick] (B) -- (C);
  \draw [thick] (D) -- (A);
\draw [thick] (D) -- (B);
\draw [thick] (E) -- (C);
\draw [thick] (E) -- (B);
\draw [thick] (E) -- (F);
\draw [thick] (D) -- (F);
\draw [thick] (D) -- (E);

\fill [fill opacity=0.25][gray!100!white] (A) -- (B) -- (D) -- cycle;
\fill [fill opacity=0.25][gray!100!white] (B) -- (C) -- (E) -- cycle;
\fill [fill opacity=0.25][gray!100!white] (D) -- (E) -- (F) -- cycle;

\fill (7,0) circle (2.5pt);
\fill (8.5,0) circle (2.5pt);
\fill (10,0) circle (2.5pt);
\fill (7.75,2/2) circle (2.5pt); 
\fill (6.5/2+6,2/2) circle (2.5pt); 
\fill (5/2+6,4/2) circle (2.5pt); 

 \end{tikzpicture}
\end{center}

 \begin{center}
\begin{tikzpicture}
\coordinate [label=below left:$v_0$]    (A) at (2/2,0); 
 \coordinate [label=below right:$v_1$]   (B) at (5/2,0); 
 \coordinate  [label=below right:$v_2$]   (C) at (8/2,0); 
\coordinate  [label=left:$v_3$]   (D) at (3.5/2,2/2); 
\coordinate  [label=right:$v_4$]   (E) at (6.5/2,2/2); 
\coordinate  [label=right:$v_5$]   (F) at (5/2,4/2);

\fill (1,0) circle (2.5pt);
\fill (2.5,0) circle (2.5pt);
\fill (4,0) circle (2.5pt);
\fill (1.75,2/2) circle (2.5pt); 
\fill (6.5/2,2/2) circle (2.5pt); 
\fill (5/2,4/2) circle (2.5pt);

 \coordinate[label=left:$\mathcal{H}'$:] (G) at (0.5/2,2/2);
 \draw [dashed,thick] (A) -- (B);
 \draw [dashed,thick] (B) -- (C);
  \draw [dashed,thick] (D) -- (A);
\draw [dashed,thick] (D) -- (B);
\draw [dashed,thick] (E) -- (C);
\draw [dashed,thick] (E) -- (B);
\draw [dashed,thick] (E) -- (F);
\draw [dashed,thick] (D) -- (F);
\draw [dashed,thick] (D) -- (E);

\fill [fill opacity=0.25][gray!100!white] (A) -- (B) -- (D) -- cycle;
\fill [fill opacity=0.25][gray!100!white] (B) -- (C) -- (E) -- cycle;
\fill [fill opacity=0.25][gray!100!white] (D) -- (E) -- (F) -- cycle;
\fill [fill opacity=0.25][gray!100!white] (D) -- (E) -- (B) -- cycle;

\coordinate [label=below left:$v_0$]    (A) at (1+6,0); 
 \coordinate [label=below right:$v_1$]   (B) at (2.5+6,0); 
 \coordinate  [label=below right:$v_2$]   (C) at (4+6,0); 
\coordinate  [label=left:$v_3$]   (D) at (1.75+6,2/2); 
\coordinate  [label=right:$v_4$]   (E) at (6.5/2+6,2/2); 
\coordinate  [label=right:$v_5$]   (F) at (5/2+6,4/2); 

 \coordinate[label=left:$\Delta{\mathcal{H}'}$:] (G) at (0.5/2+6,2/2);
 \draw [thick] (A) -- (B);
 \draw [thick] (B) -- (C);
  \draw [thick] (D) -- (A);
\draw [thick] (D) -- (B);
\draw [thick] (E) -- (C);
\draw [thick] (E) -- (B);
\draw [thick] (E) -- (F);
\draw [thick] (D) -- (F);
\draw [thick] (D) -- (E);

\fill [fill opacity=0.25][gray!100!white] (A) -- (B) -- (D) -- cycle;
\fill [fill opacity=0.25][gray!100!white] (B) -- (C) -- (E) -- cycle;
\fill [fill opacity=0.25][gray!100!white] (D) -- (E) -- (F) -- cycle;
\fill [fill opacity=0.25][gray!100!white] (D) -- (E) -- (B) -- cycle;

\fill (7,0) circle (2.5pt);
\fill (8.5,0) circle (2.5pt);
\fill (10,0) circle (2.5pt);
\fill (7.75,2/2) circle (2.5pt); 
\fill (6.5/2+6,2/2) circle (2.5pt); 
\fill (5/2+6,4/2) circle (2.5pt); 
 \end{tikzpicture}
\end{center}
\caption{Example~\ref{ex81}.}
\end{figure}
\end{example}


 For general hypergraphs $\mathcal{H}$ and $\mathcal{H}'$, 
the following example shows that  the homomorphisms (\ref{e88-1}), (\ref{e88-2}) and (\ref{e88-3})   of the homology groups  can be distinct. 

\begin{example}\label{ex1}
Let $\mathcal {H}=\{\{v_0,v_1\},\{v_1,v_2\},\{v_0,v_2\}\}$. Let $\sigma=\{v_0,v_1,v_2\}$. Let $\mathcal{H}'=\mathcal{H}\cup\{\sigma\}$.   Let $\varphi$ be the canonical inclusion of $\mathcal{H}$ into $\mathcal{H}'$. Then 
\begin{enumerate}[(a). ]
\item
$H_0(\mathcal{H})=H_0(\mathcal{H}')=0$, $H_1(\mathcal{H})=\mathbb{Z}$, and $H_1(\mathcal{H}')=0$. Moreover,  $\varphi_*$ is
\begin{itemize}
\item
 the identity map from zero to zero for the homology groups of  dimension $0$, 
 \item
  the zero map on $\mathbb{Z}$  for the homology  groups of  dimension $1$;
  \end{itemize}
\item
$\delta \mathcal{H}=\delta\mathcal{H}'=\emptyset$. Moreover, $(\delta\varphi)_*$ is
 the identity map from zero to zero  for the homology groups of  all   dimensions;
\item
$\Delta \mathcal{H}\simeq S^1$ and $\Delta\mathcal{H}'\simeq *$.  Moreover, $(\Delta\varphi)_*$ is
\begin{itemize}
\item
 the identity map on $\mathbb{Z}$ for the homology groups of   dimension $0$, 
 \item
 the zero map on $\mathbb{Z}$ for the homology groups of   dimension $1$. 
 \end{itemize}  
\end{enumerate} 
\end{example}



    \section{Discrete Morse Functions  on Hypergraphs and Critical Hyperedges}\label{s3}
    
    Let $\mathcal{H}$  be a hypergraph on $V$.     In this section, we define the discrete  Morse functions on $\mathcal{H}$ and their critical hyperedges.  We study the extensions  of discrete Morse functions on hypergraphs to discrete Morse functions on the associated simplicial complexes. 
    
    \smallskip

\subsection{ Discrete Morse Functions on Hypergraphs }

    \begin{definition}\label{def1}
   A function $f: \mathcal{H}\longrightarrow \mathbb{R}$ is called a {\it discrete Morse function} on $\mathcal{H}$ if for every $n\geq 0$ and every  $\alpha^{(n)}\in \mathcal{H}$,  both of the two conditions hold:
    \begin{enumerate}[(i).]
    \item
    $\#\{\beta^{(n+1)}>\alpha^{(n)}\mid f(\beta)\leq f(\alpha), \beta\in \mathcal{H}\}\leq 1$;
    \item
    $\#\{\gamma^{(n-1)}<\alpha^{(n)}\mid f(\gamma)\geq f(\alpha), \gamma\in \mathcal{H}\}\leq 1$. 
    \end{enumerate}
    \end{definition}
    
    \begin{remark}
   Note that Definition~\ref{def1}  is a generalization of \cite[Definition~2.1]{forman1}.   In particular, we  take  $\mathcal{H}$ to be  a simplicial complex  in Definition~\ref{def1}.  Then the discrete Morse functions defined in Definition~\ref{def1} are the same as the discrete Morse functions defined in \cite[Definition~2.1]{forman1}.  
   \end{remark}
   
   The next lemma follows from Definition~\ref{def1}.

\begin{lemma}\label{lem1}
Let $\mathcal{H}$ and $\mathcal{H}'$ be two hypergraphs such that $\mathcal{H}'\subseteq \mathcal{H}$. Suppose $f:\mathcal{H}\longrightarrow\mathbb{R}$ is a discrete Morse function on $\mathcal{H}$. Let $f'=f\mid_{\mathcal{H}'}$  be  the restriction of $f$ to $\mathcal{H}'$.  Then $f'$ is a discrete Morse function on $\mathcal{H}'$.   
\end{lemma}
\begin{proof}
For every $\alpha^{(n)}\in \mathcal{H}$, since $\mathcal{H}'\subseteq \mathcal{H}$, 
\begin{eqnarray*}
\{\beta^{(n+1)}>\alpha^{(n)} \mid f'(\beta)\leq f'(\alpha), \beta\in\mathcal{H}'\}\subseteq \{\beta^{(n+1)}>\alpha^{(n)} \mid f(\beta)\leq f(\alpha), \beta\in\mathcal{H}\}. 
\end{eqnarray*}
Since $f$ is a discrete Morse function on $\mathcal{H}$,
\begin{eqnarray*}
\#\{\beta^{(n+1)}>\alpha^{(n)} \mid f'(\beta)\leq f'(\alpha), \beta\in\mathcal{H}'\}\leq \# \{\beta^{(n+1)}>\alpha^{(n)} \mid f(\beta)\leq f(\alpha), \beta\in\mathcal{H}\}\leq 1. 
\end{eqnarray*}
Similarly, 
\begin{eqnarray*}
\#\{\gamma^{(n-1)}<\alpha^{(n)} \mid f'(\gamma)\geq f'(\alpha), \gamma\in\mathcal{H}'\}\leq \# \{\gamma^{(n-1)}<\alpha^{(n)} \mid f(\gamma)\geq f(\alpha), \gamma\in\mathcal{H}\}\leq 
1. 
\end{eqnarray*}
Hence $f'$ is a discrete Morse function on $\mathcal{H}'$. 
\end{proof}

\begin{remark}
In particular, if $\mathcal{H}$ is a simplicial complex, then Lemma~\ref{lem1} is reduced to \cite[Lemma~2.1]{forman1}.  
\end{remark}

The next proposition follows from Lemma~\ref{lem1} immediately. 

\begin{proposition}\label{pr1}
\begin{enumerate}[(i).]
\item
Let $\overline f: \Delta\mathcal{H}\longrightarrow\mathbb{R}$ be a discrete Morse function on $\Delta\mathcal{H}$. Then $f=\overline f\mid_{\mathcal{H}}$ is a discrete Morse function on $\mathcal{H}$;
\item
Let $f:\mathcal{H}\longrightarrow \mathbb{R}$ be a discrete Morse function on $\mathcal{H}$. Then ${\underline   f}=f\mid _{\delta\mathcal{H}}$ is a discrete Morse function on $\delta\mathcal{H}$.  
\end{enumerate}
\end{proposition}
\begin{proof}
We  apply Lemma~\ref{lem1}  to the pair $\mathcal{H}\subseteq \Delta\mathcal{H}$.  We  obtain (i).  We  apply Lemma~\ref{lem1}  to the pair $\delta\mathcal{H}\subseteq  \mathcal{H}$.  We  obtain (ii). 
\end{proof}

  The next corollary is a consequence of \cite{forman1} and Proposition~\ref{pr1}. 

\begin{corollary}\label{cor1}
For any hypergraph $\mathcal{H}$, there exist discrete Morse functions $\overline f$ on $\Delta\mathcal{H}$, $f$ on $\mathcal{H}$, and ${\underline   f}$ on $\delta\mathcal{H}$ such that $f=\overline f\mid_{\mathcal{H}}$  and ${\underline   f}= f\mid_{\delta\mathcal{H}}=\overline f\mid_{\delta\mathcal{H}}$.   
\end{corollary}
\begin{proof}
By \cite[Section~4]{forman1},  there  exists a discrete Morse function $\overline f$  on $\Delta\mathcal{H}$ (in the sense of \cite[Definition~2.1]{forman1}).  By Proposition~\ref{pr1},  we  obtain a discrete Morse function $f$  on $\mathcal{H}$ (in the sense of Definition~\ref{def1})  and a discrete Morse function ${\underline   f}$ on $\delta\mathcal{H}$ (in the sense of \cite[Definition~2.1]{forman1}) such that  $f=\overline f\mid_{\mathcal{H}}$  and ${\underline   f}= f\mid_{\delta\mathcal{H}}=\overline f\mid_{\delta\mathcal{H}}$.  
 \end{proof}


\smallskip

\subsection{Critical Hyperedges}\label{subs-3.2}


Let $f$ be a discrete Morse function on $\mathcal{H}$. 

\begin{definition}\label{def2}
A hyperedge $\alpha^{(n)}\in\mathcal{H}$ is called {\it critical} if  
both of the following two conditions hold:
\begin{enumerate}[(i).]
\item
$\#\{\beta^{(n+1)}>\alpha^{(n)}\mid f(\beta)\leq f(\alpha),\beta\in\mathcal{H}\}=0$;
\item
$\#\{\gamma^{(n-1)}<\alpha^{(n)}\mid f(\gamma)\geq f(\alpha),\gamma\in\mathcal{H}\}=0$. 
\end{enumerate}
\end{definition}

\begin{remark}
    In particular, if $\mathcal{H}$ is a simplicial complex, then the critical hyperedges defined in Definition~\ref{def2} are the same as the critical simplices defined in \cite[Definition~2.2]{forman1}. 
\end{remark}

\begin{definition}\label{def222}
 We use $M(f,\mathcal{H})$ to denote the set of all critical hyperedges.   
\end{definition}

The next lemma is equivalent to  Definition~\ref{def2}.  

\begin{lemma}\label{le-ab}
For any $n\geq 0$ and any $\alpha^{(n)}\in\mathcal{H}$,  we  have that $\alpha\notin M(f,\mathcal{H})$  if at least one of the following conditions hold:

\begin{enumerate}[(A).]
\item
there exists $\beta^{(n+1)}>\alpha^{(n)}$, $\beta\in\mathcal{H}$, such that $f(\beta)\leq f(\alpha)$;
\item
there exists $\gamma^{(n-1)}<\alpha^{(n)}$, $\gamma\in\mathcal{H}$, such that $f(\gamma)\geq f(\alpha)$. 
\end{enumerate}
\end{lemma}

\begin{proof}
The lemma is the contrapositive statement of Definition~\ref{def2}.  
\end{proof}

\begin{lemma}\label{lem2}
Let $\mathcal{H}$ and $\mathcal{H}'$ be two hyperedges such that $\mathcal{H}'\subseteq\mathcal{H}$. Let $f:\mathcal{H}\longrightarrow\mathbb{R}$ be a discrete Morse function on $\mathcal{H}$ and let $f'=f\mid_{\mathcal{H}'}$.  Then 
\begin{eqnarray*}
M(f,\mathcal{H})\cap\mathcal{H}'\subseteq M(f',\mathcal{H}').
\end{eqnarray*}  
\end{lemma}

\begin{proof}
The lemma follows from a  straight-forward verification by using Definition~\ref{def2}. 
Let $\alpha^{(n)}\in M(f,\mathcal{H})\cap\mathcal{H}'$.  Then since $\alpha^{(n)}\in M(f,\mathcal{H})$, by Definition~\ref{def2}, 
\begin{eqnarray*}
\#\{\beta^{(n+1)}>\alpha^{(n)}\mid  f(\beta)\leq f(\alpha), \beta\in\mathcal{H}\}=0
\end{eqnarray*}
and 
\begin{eqnarray*}
\#\{\gamma^{(n-1)}<\alpha^{(n)}\mid  f(\gamma)\geq f(\alpha), \gamma\in\mathcal{H}\}=0. 
\end{eqnarray*}
Since $\mathcal{H}'\subseteq \mathcal{H}$, it follows that
\begin{eqnarray*}
\#\{\beta^{(n+1)}>\alpha^{(n)}\mid  f'(\beta)\leq f'(\alpha), \beta\in\mathcal{H}'\}=0
\end{eqnarray*}
and 
\begin{eqnarray*}
\#\{\gamma^{(n-1)}<\alpha^{(n)}\mid  f'(\gamma)\geq f'(\alpha), \gamma\in\mathcal{H}'\}=0. 
\end{eqnarray*}
Therefore, we obtain $\alpha\in M(f',\mathcal{H}')$. The lemma is proved. 
\end{proof}

The next proposition follows from Lemma~\ref{lem2} immediately. 

\begin{proposition}\label{pr-1.88}
Let $\overline f: \Delta\mathcal{H}\longrightarrow \mathbb{R}$ be a discrete Morse function on $\Delta\mathcal{H}$.  Let $f: \mathcal{H}\longrightarrow \mathbb{R}$  and $\underline  f: \delta\mathcal{H}\longrightarrow\mathbb{R}$ be the discrete Morse functions induced from $\overline f$.  Then 
\begin{eqnarray}
M(\overline f,\Delta\mathcal{H})\cap \mathcal{H}&\subseteq& M(f,\mathcal{H}), \label{eq-1.z}\\
M(\overline f,\Delta\mathcal{H})\cap\delta\mathcal{H}&\subseteq& M({\underline   f},\delta\mathcal{H}),\nonumber\\
M(f,\mathcal{H})\cap \delta\mathcal{H}&\subseteq& M({\underline   f},\delta\mathcal{H}). \nonumber 
\end{eqnarray}
\end{proposition}

\begin{proof}
We  apply Lemma~\ref{lem2}  to the pair $\mathcal{H}\subseteq \Delta\mathcal{H}$.  Then we obtain (\ref{eq-1.z}).  Similarly,  we apply Lemma~\ref{lem2}  to the pairs $\delta\mathcal{H}\subseteq  \Delta\mathcal{H}$  and    $\delta\mathcal{H}\subseteq \mathcal{H}$.  Then we obtain the other two subset relations.  
\end{proof}

\smallskip

\subsection{Extensions of   Discrete Morse Functions}

The next lemma is proved in \cite{forman1}.  

\begin{lemma}{\rm (cf.   \cite[Lemma~2.5]{forman1})}.
\label{le-rev999}
Let $\mathcal{K}$  be a simplicial complex.  Let $g$  be  a  discrete Morse function on $\mathcal{K}$.  Let $\alpha$  be a simplex of $\mathcal{K}$.  Then the conditions \begin{enumerate}[(A).]
\item
there exists $\beta^{(n+1)}>\alpha^{(n)}$, $\beta\in\mathcal{K}$, such that $g(\beta)\leq g(\alpha)$;
\item
there exists $\gamma^{(n-1)}<\alpha^{(n)}$, $\gamma\in\mathcal{K}$, such that $g(\gamma)\geq g(\alpha)$ 
\end{enumerate} 
cannot both be true.   \qed
\end{lemma}

The next  lemma is a consequence of Lemma~\ref{le-rev999}.  

\begin{lemma}\label{le-rev731}
Let $\mathcal{H}$  be a hypergraph.  Let $\overline f$  be a discrete Morse function on $\Delta\mathcal{H}$.    Let $\alpha\in \Delta\mathcal{H}$.   Then the conditions
\begin{enumerate}[(A).]
\item
there exists $\beta^{(n+1)}>\alpha^{(n)}$, $\beta\in\Delta\mathcal{H}$, such that $\overline f(\beta)\leq \overline f(\alpha)$;
\item
there exists $\gamma^{(n-1)}<\alpha^{(n)}$, $\gamma\in\Delta\mathcal{H}$, such that $\overline f(\gamma)\geq \overline f(\alpha)$. 
\end{enumerate}
 cannot both be true.   
\end{lemma}
\begin{proof}
In Lemma~\ref{le-rev999},  we let $\mathcal{K}$ be $\Delta\mathcal{H}$  and   let $g$  be $\overline f$.   The lemma follows.  
\end{proof}

The next proposition follows from Lemma~\ref{le-rev731}. 

\begin{proposition}[Obstructions for The Extensions of   Discrete Morse Functions]  
Let $\mathcal{H}$  be a hypergraph.  Let $f$  be a discrete Morse function on $\mathcal{H}$.  If there exists $\alpha\in \mathcal{H}$  such that both (A) and (B)  in Lemma~\ref{le-ab}   hold for $\alpha$,  then $f$ cannot be extended to be a discrete Morse function $\overline f$  on  $\Delta\mathcal{H}$.  
\label{pr-rev012}
\end{proposition}

\begin{proof}
Suppose to the contrary  that $f$ can  be extended to be a discrete Morse function $\overline f$  on  $\Delta\mathcal{H}$.  Then  by Lemma~\ref{le-rev731},  for any $\alpha\in \Delta\mathcal{H}$,  the conditions 
(A) and (B)  in Lemma~\ref{le-rev731} 
  cannot both be true.  Note that for any $\alpha\in \mathcal{H}$,   we  have $f(\alpha)=\overline f(\alpha)$.    Thus for any $\alpha\in  \mathcal{H}$,   the conditions (A) and (B)  in Lemma~\ref{le-ab}  cannot both be true.  This contradicts with our assumption that  there exists $\alpha\in \mathcal{H}$  such that both (A) and (B)  in Lemma~\ref{le-ab}   hold for $\alpha$.  Therefore,  $f$ cannot be extended to be a discrete Morse function $\overline f$  on  $\Delta\mathcal{H}$. 
\end{proof}

The next is an example for Proposition~\ref{pr-rev012}. 

\begin{example}\label{ex-2.a}
Let $\mathcal{H}=\{\{v_0\},\{v_0,v_1\},\{v_0,v_1,v_2\}\}$. Let 
\begin{eqnarray*}
f(\{v_0\})=2,  \text{\ \ \ }
f(\{v_0,v_1\})=1,   \text{\ \ \ }
f(\{v_0,v_1,v_2 \})=0.  
\end{eqnarray*}
Then $f$ is a discrete Morse function on $\mathcal{H}$, and $\{v_0,v_1\}$ is not critical. 
\begin{enumerate}[(i).]
\item
The hyperedge $\{v_0,v_1\}$ satisfies both of the conditions (A) and (B); 
\item
The discrete Morse function
$f$ cannot be extended to a discrete Morse function on $\Delta\mathcal{H}$. 
\end{enumerate}
\end{example}

 The next proposition proves that  under certain conditions,  the obstruction for the extensions of   discrete Morse functions
given in Proposition~\ref{pr-rev012} would fail.  

\begin{proposition}\label{le-2.a}
Suppose $\mathcal{H}$ satisfies the following condition 
\begin{quote}
{\bf Condition (C)}. for any $n\geq 1$ and any hyperedges $\beta^{(n+1)}>\alpha^{(n)}>\gamma^{(n-1)}$ of $\mathcal{H}$, there exists $\hat   \alpha^{(n)}\in\mathcal{H}$, $\hat   \alpha\neq \alpha$,  such that $\beta>\hat  \alpha>\gamma$. 
\end{quote}
Then the conditions (A) and (B)  in Lemma~\ref{le-ab}    cannot both be true. 
\end{proposition}

\begin{proof}
The proof is similar with the proof of    \cite[Lemma~2.5]{forman1}.  Let $\mathcal{H}$  be a hypergraph satisfying the  condition (C).  Let $f$  be a discrete Morse function on $\mathcal{H}$.  Suppose $(A)$  in Lemma~\ref{le-ab}  is true.    Then  by Definition~\ref{def1},  we  have 
\begin{eqnarray}\label{eq-rev-obs1}
f(\beta)>f(\hat\alpha)
\end{eqnarray}
  for any $\hat   \alpha^{(n)}\in\mathcal{H}$  with  $\hat   \alpha\neq \alpha$  and  $\beta>\hat  \alpha>\gamma$.  By (\ref{eq-rev-obs1})  and $(A)$  in Lemma~\ref{le-ab},  we  have 
  \begin{eqnarray}\label{eq-rev-cont1}
  f(\hat\alpha)< f(\alpha).  
  \end{eqnarray}
Now suppose $(B)$  in Lemma~\ref{le-ab}  is true.  Then  by Definition~\ref{def1},  we  have 
\begin{eqnarray}\label{eq-rev-obs2}
f(\gamma)<f(\hat\alpha)
\end{eqnarray}
 for any $\hat   \alpha^{(n)}\in\mathcal{H}$  with  $\hat   \alpha\neq \alpha$  and  $\beta>\hat  \alpha>\gamma$.  By (\ref{eq-rev-obs2})  and $(B)$  in Lemma~\ref{le-ab},  we  have 
  \begin{eqnarray}\label{eq-rev-cont2}
  f(\hat\alpha)> f(\alpha).  
  \end{eqnarray}
Note that (\ref{eq-rev-cont1}) and (\ref{eq-rev-cont2})  contradict with each other.  Therefore,   the conditions (A) and (B) cannot both be true.   
\end{proof}

There do exist hypergraphs $\mathcal{H}$  and discrete Morse functions $f$  on $\mathcal{H}$ such that the condition (C)  in Proposition~\ref{le-2.a} does not hold while   
the conditions (A) and (B)  in Lemma~\ref{le-ab}    cannot both be true.
We  consider the next example.  

\begin{example}\label{ex-a.a}
Let $\mathcal{H}=\{\{v_0\},\{v_1\},\{v_2\}, \{v_0,v_1,v_2\}\}$. Let $f(\{v_0\})=f(\{v_1\})=f(\{v_2\})=2$,  $f(\{v_0,v_1,v_2\})=0$.   Then $f$ is a discrete Morse function on $\mathcal{H}$.  All the hyperedges are critical.  The condition (C)  in Proposition~\ref{le-2.a} does not hold while   for any $\alpha\in\mathcal{H}$,  
the conditions (A) and (B)  in Lemma~\ref{le-ab}    cannot both be true.
\end{example}

\begin{proposition}
In Example~\ref{ex-a.a},   $f$ cannot be extended to be a discrete Morse function on $\Delta\mathcal{H}$. 
\end{proposition}

\begin{proof} 
Suppose to the contrary, $\overline f: \Delta\mathcal{H}\longrightarrow \mathbb{R}$ is a discrete Morse function such that $f=\overline f\mid_{\mathcal{H}}$.  Then there exists at least two edges among $\{v_0,v_1\}$, $\{v_1,v_2\}$ and $\{v_0,v_2\}$, say $\{v_0,v_1\}$ and $\{v_1,v_2\}$,  such that 
\begin{eqnarray*}
\overline f( \{v_0,v_1\})<\overline f (\{v_0,v_1,v_2\}),\text{\ \ \ } \overline f( \{v_1,v_2\}) < \overline f (\{v_0,v_1,v_2\}).
\end{eqnarray*} 
Moreover,  we  have both of the followings
\begin{itemize}
\item
$\overline f(u)<  \overline f( \{v_0,v_1\})$ for $u=v_0$ or $v_1$; 
\item
$\overline f(w)<  \overline f( \{v_1,v_2\})$ for $w=v_1$ or $v_2$. 
\end{itemize}
 Hence there exists at least one $0$-hyperedge  $\{x\}$  among $\{v_0\}$, $\{v_1\}$ and $\{v_2\}$ where $\overline f (\{x\})$ is smaller than $\overline f (\{v_0,v_1,v_2\})$.  This contradicts that $\overline f$ is the extension of $f$.  
\end{proof}


\smallskip

\section{Discrete Gradient Vector Fields  on Hypergraphs}\label{s4}

 Let $\mathcal{H}$  be a hypergraph.  In this section, we study the discrete gradient vector fields on $\mathcal{H}$. 

\smallskip

\subsection{Abstract Discrete Gradient Vector Fields on Hypergraphs}\label{ss4.1-r}

\begin{definition}\label{def-a9}
 An {\it (abstract)  discrete gradient vector field} $\mathbb{V}$ on $\mathcal{H}$ is a collection of pairs $\{\alpha^{(n)}<\beta^{(n+1)}\}$   of hyperedges of $\mathcal{H}$,  $n\geq 0$,  such that  there is no non-trivial closed paths of the form
\begin{eqnarray*}
\alpha_0^{(n)}, \beta_0^{(n+1)}, \alpha_1^{(n)}, \beta_1^{(n+1)}, \alpha_2^{(n)}, \cdots, \alpha_r^{(n)}, \beta_r^{(n+1)}, \alpha_{r+1}^{(n)}=\alpha_0^{(n)},
\end{eqnarray*}
where for each $0\leq i\leq r$,   we  have  $\{\alpha_i^{(n)}<\beta_i^{(n+1)}\}\in \mathbb{V}$,  $\{\alpha_{i+1}^{(n)} <\beta_i^{(n+1)}\}\in \mathbb{V}$  and $\alpha_i\neq \alpha_{i+1}$. 
\end{definition}

\begin{definition}\label{def-18rev}
Let $\mathbb{V}$   be a discrete gradient vector field on $\mathcal{H}$.  If 
there does not exist any $n\geq 1$  and any triple $(\gamma^{(n-1)},\alpha^{(n)},  \beta^{(n+1)})$  of hyperedges in  $\mathcal{H}$  such that 
\begin{eqnarray*}
\gamma^{(n-1)}<\alpha^{(n)}<\beta^{(n+1)}
\end{eqnarray*}
and 
\begin{eqnarray*}
\{\gamma<\alpha\} \in \mathbb{V}~~~{\rm and }~~~  \{\alpha<\beta\}\in \mathbb{V},  
\end{eqnarray*}
then we call $\mathbb{V}$ a {\it semi-proper} discrete gradient vector field. 
\end{definition}

\begin{definition}\label{def-17rev}
Let $\mathbb{V}$   be a discrete gradient vector field on $\mathcal{H}$.  If 
each hyperedge of $\mathcal{H}$ is in at most one pair in $\mathbb{V}$, then we call $\mathbb{V}$ a {\it proper} discrete gradient vector field. 
\end{definition}

\begin{remark}
By Definition~\ref{def-18rev}  and  Definition~\ref{def-17rev},  it is direct to see that  a  proper discrete gradient vector field is semi-proper while  semi-proper discrete gradient vector field may not be proper.  
\end{remark}

Let $\mathbb{V}$   be a discrete gradient vector field on $\mathcal{H}$.  

\begin{definition}\label{def-lm}
 We define an $R$-linear map \begin{eqnarray}\label{eq-3.a}
R(\mathbb{V}):~~~ R( \mathcal{H})_n\longrightarrow R(\mathcal{H})_{n+1}  
\end{eqnarray} 
induced from $\mathbb{V}$  where for each $\alpha^{(n)}\in \mathcal{H}$,  the $R$-linear map $R(\mathbb{V})$ sends $\alpha$ to 
\begin{itemize}
\item
  $-\langle\partial\beta,\alpha\rangle\beta$ if $\{\alpha^{(n)}<\beta^{(n+1)}\}\in \mathbb{V}$, 
\item
 $0$ if $\{\alpha^{(n)}<\beta^{(n+1)}\}\notin V$. 
\end{itemize}
\end{definition}
\begin{remark}
 In Definition~\ref{def-lm},  the notion  $\langle\partial\beta,\alpha\rangle$ is the incidence number  of $\beta$ and $\alpha$ in the chain complex $C_*(\Delta\mathcal{H};R)$ (cf. \cite[p. 98]{forman1}).  Note  that  $\langle\partial\beta,\alpha\rangle$ takes the values $\pm 1$ where $1$ refers to the multiplicative unity of $R$. 
\end{remark}

\begin{lemma}\label{le-rev-xxx}
Let $\mathbb{V}$   be a discrete gradient vector field on $\mathcal{H}$.   Then  $\mathbb{V}$ is semi-proper if  and only if  the induced $R$-linear map $R(\mathbb{V})$ satisfies
\begin{eqnarray}\label{eq-rev-b2}
R(\mathbb{V})\circ R(\mathbb{V})=0.
\end{eqnarray}
\end{lemma}

\begin{proof}
Let $\mathbb{V}$  be a discrete  gradient vector field on $\mathcal{H}$.  Let  $ R(\mathbb{V})$  be the induced $R$-linear map.

($\Longrightarrow$):   Suppose $\mathbb{V}$  is semi-proper.  In order to prove  (\ref{eq-rev-b2}),  we take    an  arbitrary  $\alpha^{(n)} \in \mathcal{H}$.  By the linear property of $R(\mathbb{V})\circ R(\mathbb{V})$,  it suffices to prove 
\begin{eqnarray}\label{eq-rev-b1}
 R(\mathbb{V})\circ R(\mathbb{V})(\gamma^{(n)})=0.  
\end{eqnarray}

{\sc Case~1}.  $R(\mathbb{V})(\gamma^{(n)})=0$.  

Then  (\ref{eq-rev-b1}) follows immediately.

{\sc Case~2}.  $R(\mathbb{V})(\gamma^{(n)})\neq 0$.

Suppose  $\alpha^{(n+1)}\in \mathcal{H}$  such that $\{\gamma^{(n)}<\alpha^{(n+1)}\}\in \mathbb{V}$ 
  and  $\langle\partial\alpha,\gamma\rangle\neq 0$.  
Since $\mathbb{V}$  is semi-proper   and    $\{\gamma^{(n)}<\alpha^{(n+1)}\}\in \mathbb{V}$,  we  have that $\alpha$  is not in any other pairs of the form $\{\alpha^{(n+1)}<\beta^{(n+2)}\}$ in $V$.  Thus we have   $R(\mathbb{V})(\alpha)=0$.   Consequently, 
\begin{eqnarray*}
R(\mathbb{V})\circ R(\mathbb{V})(\gamma^{(n)})&=& R(\mathbb{V})\Big(\sum_{\{\gamma^{(n)}<\alpha^{(n+1)}\}\in \mathbb{V},\atop \langle\partial\alpha,\gamma\rangle\neq 0} -\langle\partial\alpha,\gamma\rangle\alpha\Big)\\
&=&\sum_{\{\gamma^{(n)}<\alpha^{(n+1)}\}\in \mathbb{V},\atop \langle\partial\alpha,\gamma\rangle\neq 0} -\langle\partial\alpha,\gamma\rangle R(\mathbb{V})(\alpha)\\
&=&0.  
\end{eqnarray*}

Summarizing both Case~1 and Case~2, we  obtain (\ref{eq-rev-b1}).   This  implies  (\ref{eq-rev-b2}).

($\Longleftarrow$):   Suppose $R(\mathbb{V})$ satisfies  (\ref{eq-rev-b2}).  In order to prove that $\mathbb{V}$  is semi-proper,  we suppose to the contrary that $\mathbb{V}$  is not semi-proper.  Then by Definition~\ref{def-18rev},  there exist  $n\geq 1$ and a triple   triple $(\gamma^{(n-1)},\alpha^{(n)},  \beta^{(n+1)})$  of hyperedges in  $\mathcal{H}$  such that 
\begin{eqnarray*}
\gamma^{(n-1)}<\alpha^{(n)}<\beta^{(n+1)}
\end{eqnarray*}
and 
\begin{eqnarray*}
\{\gamma<\alpha\} \in \mathbb{V}~~~{\rm and }~~~  \{\alpha<\beta\}\in \mathbb{V}.    
\end{eqnarray*}
Consequently,  
\begin{eqnarray*}
\langle R(\mathbb{V})\circ R(\mathbb{V})(\gamma^{(n-1)}),\beta^{(n+1)}\rangle&=&-\langle\partial\alpha,\gamma\rangle  \langle  R(\mathbb{V})(\alpha^{(n)}),\beta^{(n+1)}\rangle\\
&=&\langle\partial\alpha,\gamma\rangle\langle\partial\beta,\alpha\rangle \langle  \beta^{(n+1)},\beta^{(n+1)}\rangle\\
&=&\langle\partial\alpha,\gamma\rangle\langle\partial\beta,\alpha\rangle\\
& =& \pm 1.  
\end{eqnarray*}
This contradicts with our assumption that $R(\mathbb{V})$ satisfies  (\ref{eq-rev-b2}).  
Therefore,  $\mathbb{V}$  is semi-proper.  
\end{proof}
 \begin{remark}
 In  particular,  if $\mathcal{H}$ is a simplicial complex, then  the $R$-linear map  (\ref{eq-3.a}) is defined in   \cite[Definition~6.1]{forman1}. 
\end{remark}

\smallskip

\subsection{Discrete Gradient Vector Fields of Discrete Morse Functions on Hypergraphs}

Let $f$ be a discrete Morse function on $\mathcal{H}$.  

\begin{definition}\label{def33333}
The {\it discrete gradient vector field of $f$},  denoted as  $\grad   f$,    is the  collection of all  the pairs $\{\alpha^{(n)}<\beta^{(n+1)}\}$ ($n\geq 0$)   of the hyperedges in $\mathcal{H}$  such that    $f(\beta)\leq f(\alpha)$.  
\end{definition}

\begin{lemma}
Let $f$  be a discrete Morse function on $\mathcal{H}$.  Then  $\grad   f$ is a discrete gradient vector field on $\mathcal{H}$.  
\end{lemma}
\begin{proof}
By Definition~\ref{def-a9},  to prove that  $\grad   f$ is a discrete gradient vector field on $\mathcal{H}$,  we only need to prove that there is no nontrivial closed paths of the form
\begin{eqnarray*}
\alpha_0^{(n)}, \beta_0^{(n+1)}, \alpha_1^{(n)}, \beta_1^{(n+1)}, \alpha_2^{(n)}, \cdots, \alpha_r^{(n)}, \beta_r^{(n+1)}, \alpha_{r+1}^{(n)}=\alpha_0^{(n)},
\end{eqnarray*}
where for each $0\leq i\leq r$,  we  have  $\{\alpha_i^{(n)}<\beta_i^{(n+1)}\}\in \grad   f$,  $\{\alpha_{i+1}^{(n)} <\beta_i^{(n+1)}\}\in \grad   f$  and $\alpha_i\neq \alpha_{i+1}$.  Suppose to the contrary,  there exists such a closed path.  Then  by  Definition~\ref{def33333},  we  have 
\begin{eqnarray*}
f(\beta_i)\leq f(\alpha_i) ~~~~~~{\rm and}~~~~~~ f(\beta_i)\leq f(\alpha_{i+1}).  
\end{eqnarray*} 
This contradicts with Definition~\ref{def1}.  Therefore,  there is no such closed paths,  which implies that  $\grad   f$ is a discrete gradient vector field on $\mathcal{H}$.  
\end{proof}

\begin{lemma}\label{def3}
Let $f$  be a discrete Morse function on $\mathcal{H}$.  Then for any $n\geq 0$ and any  $\alpha^{(n)}\in\mathcal{H}$,  the induced $R$-linear map $R(\grad   f)$  of $\grad   f$  is given as follows: 
\begin{enumerate}[(i).]
\item
If there exists $\beta^{(n+1)}>\alpha^{(n)}$, $\beta\in \mathcal{H}$,  such that $f(\beta)\leq f(\alpha)$, then  
\begin{eqnarray*}
R(\grad   f)(\alpha)=-\langle\partial\beta,\alpha\rangle\beta; 
\end{eqnarray*}
\item
If there does not exist any $\beta^{(n+1)}>\alpha^{(n)}$, $\beta\in \mathcal{H}$,  such that $f(\beta)\leq f(\alpha)$, then 
  \begin{eqnarray*}
  R(\grad   f)(\alpha)=0.
  \end{eqnarray*}
\end{enumerate}
\end{lemma}

\begin{proof}
The lemma follows directly from Definition~\ref{def-lm}.  
\end{proof}
 

The next example shows that the discrete gradient vector fields of  discrete Morse functions on hypergraphs may not be semi-proper.

\begin{example}\label{ex-3.1}
Consider the hypergraph $\mathcal{H}$  and the discrete Morse function $f$ given in Example~\ref{ex-2.a}. 
Let $\grad   f$ be the discrete gradient vector field of $f$. Then 
\begin{eqnarray*}
(\grad   f)(\{v_0\})&=&\{v_0,v_1\},  \\
(\grad   f)(\{v_0,v_1\})&=& -\{v_0,v_1,v_2\}. 
\end{eqnarray*}
Hence $(\grad   f)\circ (\grad   f) (\{v_0\})= -\{v_0,v_1,v_2\}\neq 0$. 
\end{example}

Nevertheless,  the next proposition shows that by imposing the  condition (C) in Proposition~\ref{le-2.a}) on the hypergraphs,  the discrete gradient vector fields of  discrete Morse functions   must be proper.

\begin{proposition} 
Let $\mathcal{H}$ be a hypergraph satisfying the  condition (C) in  Proposition~\ref{le-2.a}.  Let $f$  be a discrete Morse function on $\mathcal{H}$.    Then $\grad   f$ is a proper discrete gradient vector field.   
\end{proposition}

\begin{proof}
  Firstly,  we prove that  $\grad   f$  is semi-proper.   By Lemma~\ref{le-rev-xxx}, it is sufficient to prove that 
  \begin{eqnarray} \label{eq-yyyrev}
R(\grad   f)\circ R(\grad   f)=0.
\end{eqnarray}
  The proof is similar with the proof of    \cite[Theorem~6.3~(1)]{forman1}.  Suppose $\gamma^{(n-1)},\alpha^{(n)}\in \mathcal{H}$  and 
  \begin{eqnarray*}
  R(\grad   f)(\gamma^{(n-1)})=\pm \alpha^{(n)}. 
  \end{eqnarray*}
  Then $\gamma^{(n-1)}<\alpha^{(n)}$  and $f(\gamma^{(n-1)})\geq f(\alpha^{(n)})$.  Since $\mathcal{H}$  satisfies the condition (C),  by Proposition~\ref{le-2.a},  
  we see that   there does not exist any $\beta^{(n+1)}\in\mathcal{H}$  such that $\alpha^{(n)}<\beta^{(n+1)}$  and $f(\beta^{(n+1)})\leq f(\alpha^{(n)})$.  Thus 
   \begin{eqnarray*}
  R(\grad   f)(\alpha^{(n)})=0, 
  \end{eqnarray*}
  which implies 
    \begin{eqnarray*} 
R(\grad   f)\circ R(\grad   f)(\gamma^{(n-1)})=0.
\end{eqnarray*}
Therefore, we obtain (\ref{eq-yyyrev}).

Secondly,  we prove that $\grad   f$  is  proper.  Suppose to the contrary,  $\grad   f$  is not proper.   Then since $\grad   f$  is  semi-proper,  it follows from Definition~\ref{def-17rev} and Definition~\ref{def-18rev}  that   at least one of the followings is satisfied:
\begin{enumerate}[(a).]
\item
There exist  $\gamma_1^{(n-1)}, \gamma_2^{(n-1)},  \alpha^{(n)}\in \mathcal{H}$  with  $\gamma_1\neq\gamma_2$    such that 
\begin{eqnarray*}
\{\gamma_1<\alpha\}\in \grad   f ~~~~~~{\rm and }~~~~~~ \{\gamma_2<\alpha\}\in \grad   f;  
\end{eqnarray*}
\item
There exist  $\beta_1^{(n+1)}, \beta_2^{(n+1)},  \alpha^{(n)}\in \mathcal{H}$ with $\beta_1\neq\beta_2$  such that 
\begin{eqnarray*}
\{\alpha<\beta_1\}\in \grad   f ~~~~~~{\rm and }~~~~~~ \{\alpha<\beta_2\}\in \grad   f. 
\end{eqnarray*}
\end{enumerate}
Equivalently,  we can re-state (a) and (b) respectively as (a)' and (b)':  
\begin{enumerate}[(a)'.]
\item
There exist  $\gamma_1^{(n-1)}, \gamma_2^{(n-1)},  \alpha^{(n)}\in \mathcal{H}$ with $\gamma_1\neq\gamma_2$   such that 
\begin{eqnarray*}
  R(\grad   f)(\gamma_1^{(n-1)})=\pm \alpha^{(n)}~~~~~~{\rm and }  ~~~~~~  R(\grad   f)(\gamma_2^{(n-1)})=\pm \alpha^{(n)}.   
\end{eqnarray*}
\item
There exist  $\beta_1^{(n+1)}, \beta_2^{(n+1)},  \alpha^{(n)}\in \mathcal{H}$ with $\beta_1\neq\beta_2$ such that 
\begin{eqnarray*}
  R(\grad   f)(\alpha^{(n)})=\pm \beta_1^{(n+1)}~~~~~~{\rm and }  ~~~~~~  R(\grad   f)(\alpha^{(n)})=\pm \beta_2^{(n+1)}.   
\end{eqnarray*}
\end{enumerate}
 Since $R(\grad   f)$  is an $R$-linear map,  we see that (b)'  is impossible.  On the other hand,  (a)'  implies that $\gamma_i<\alpha$  and $f(\alpha)\leq f(\gamma_i)$  for $i=1,2$.  
 This contradicts that $f$  is a discrete Morse function on $\mathcal{H}$.  Therefore,  $\grad   f$  must be proper.  
\end{proof}

\smallskip



\subsection{Restrictions and Extensions of Discrete Gradient Vector Fields}



Let $\mathcal{H}$ and $\mathcal{H}'$ be two hypergraphs such that $\mathcal{H}'\subseteq\mathcal{H}$.

\begin{lemma}\label{le-9.1}
 Let $\mathbb{V}$ be a discrete gradient vector field on $\mathcal{H}'$. Then $\mathbb{V}$ is also a discrete gradient vector field on $\mathcal{H}$. 
\end{lemma}

\begin{proof}
Let $\mathbb{V}$ be a   discrete gradient vector field on $\mathcal{H}'$.  We generalize the Hasse diagram construction (cf. \cite[Section~6]{forman2}) of simplicial complexes and apply a similar argument to hypergraphs.  Specifically,  
\begin{itemize}
\item
We construct a digraph $D_\mathcal{H}$ (resp. $D_{\mathcal{H}'}$) as follows:
\begin{itemize}
\item
{\sc The Vertices of $D_\mathcal{H}$ (resp. $D_{\mathcal{H}'}$)}:  the  vertices of $D_\mathcal{H}$ (resp. $D_{\mathcal{H}'}$) are in 1-1 correspondence with the hyperedges of $\mathcal{H}$ (resp. $\mathcal{H}'$);
\item
{\sc The Directed Edges  of $D_\mathcal{H}$ (resp. $D_{\mathcal{H}'}$)}:   for any hyperedges $\alpha$ and $\beta$  of $\mathcal{H}$ (resp. of $\mathcal{H}'$), we assign   a directed edge in $D_\mathcal{H}$  (resp. in $D_{\mathcal{H}'}$) from $\beta$ to $\alpha$, denoted as $\beta\to\alpha$, iff. $\alpha^{(n)}<\beta^{(n+1)}$ for some $n\geq 0$.  
  \end{itemize}
  \item
  We construct a  digraph $D_{\mathcal{H}',\mathbb{V}}$  as follows:
  
  \begin{itemize}
  
  \item
  {\sc The Vertices of $D_{\mathcal{H}',\mathbb{V}}$}:   the  vertices of $D_{\mathcal{H}',\mathbb{V}}$ are in 1-1 correspondence with the hyperedges of $\mathcal{H}'$;
  
  \item
    {\sc The Directed Edges of $D_{\mathcal{H}',\mathbb{V}}$}:   for any directed edge $\beta^{(n+1)}\to \alpha^{(n)}$ in $D_{\mathcal{H}'}$, 
    \begin{itemize}
\item
    if  $\{\alpha^{(n)},\beta^{(n+1)}\}\in \mathbb{V}$, then we reverse the direction of $\beta\to\alpha$ in $D_{\mathcal{H}'}$ and obtain a  new directed edge $\alpha\to\beta$. We assign $\alpha\to \beta$ as a directed edge of $D_{\mathcal{H}',\mathbb{V}}$;  
     
\item
if  $\{\alpha^{(n)},\beta^{(n+1)}\}\notin \mathbb{V}$,  then we assign $\beta\to \alpha$   as a directed edge of $D_{\mathcal{H}',\mathbb{V}}$. 
\end{itemize}  
  \end{itemize}
  \item
  We construct a digraph $D_{\mathcal{H},\mathbb{V}}$  as follows:
  \begin{itemize}
  \item
   {\sc The Vertices of $D_{\mathcal{H},\mathbb{V}}$}:   the  vertices of $D_{\mathcal{H},\mathbb{V}}$ are in 1-1 correspondence with the hyperedges of $\mathcal{H}$;
  \item
    {\sc The Directed Edges of $D_{\mathcal{H},\mathbb{V}}$}:  Let $E(-)$ denotes the set of directed edges of a digraph and let
\begin{eqnarray}\label{eq-888}
E(D_{\mathcal{H},\mathbb{V}})= E(D_{\mathcal{H}',\mathbb{V}})\cup  ( E(D_{\mathcal{H}})\setminus E(D_{\mathcal{H}'})).  
\end{eqnarray}
 We notice that for any directed edge $\beta\to\alpha$ in  $E(D_{\mathcal{H}})\setminus E(D_{\mathcal{H}'})$,  at least one of $\alpha$ and $\beta$ is not a vertex of $D_{\mathcal{H}'}$  (otherwise the directed edge $\beta\to\alpha$ would be in $E(D_{\mathcal{H}'})$).  Hence the union in (\ref{eq-888}) is a disjoint union. 
  \end{itemize}
  \end{itemize}

We  note that the following two statements are equivalent: 
\begin{enumerate}[(I).]
\item
there is no non-trivial  closed paths on $\mathcal{H}$  (resp. $\mathcal{H}'$) of the form in Definition~\ref{def-a9};
\item
 there is no non-trivial closed paths in the underlying graph of $D_{\mathcal{H},\mathbb{V}}$    (resp. $D_{\mathcal{H}',\mathbb{V}}$)  
 of the form
\begin{eqnarray*}
\alpha_0^{(n)}, \beta_0^{(n+1)}, \alpha_1^{(n)}, \beta_1^{(n+1)}, \alpha_2^{(n)}, \cdots, \alpha_r^{(n)}, \beta_r^{(n+1)}, \alpha_{r+1}^{(n)}=\alpha_0^{(n)},
\end{eqnarray*}
such that  for each $0\leq i\leq r$,   all of the three conditions are satisfied:
\begin{enumerate}[(a).]
\item
  $\{\alpha_i^{(n)}<\beta_i^{(n+1)}\}\in \mathbb{V}$,  
\item
   $\{\alpha_{i+1}^{(n)} <\beta_i^{(n+1)}\}\in \mathbb{V}$,  
\item    
   $\alpha_i\neq \alpha_{i+1}$.  
   \end{enumerate}
\end{enumerate}

Since $\mathbb{V}$ is  a discrete gradient vector field on $\mathcal{H}'$,  it follows from 
the definition  that there is no non-trivial  closed paths on  $\mathcal{H}'$  of the form in Definition~\ref{def-a9}.  
Thus   there is no   non-trivial closed   paths  in the underlying graph of $D_{{\mathcal{H}'},\mathbb{V}}$  satisfying the conditions (a), (b)  and (c)  in  (II).   By the construction of $D_{\mathcal{H},\mathbb{V}}$,   it follows that there is no nontrivial closed     paths  in the underlying graph of $D_{\mathcal{H},\mathbb{V}}$ satisfying the conditions (a), (b)  and (c)  in  (II) as well.  Consequently,  there is no non-trivial  closed paths on  $\mathcal{H}$  of the form in Definition~\ref{def-a9}.  
Therefore,    we  have that $\mathbb{V}$ is a discrete gradient vector field on $\mathcal{H}$. 
 \end{proof}
 
 \begin{remark}
 The last paragraph  in  proof of   Lemma~\ref{le-9.1}   is  a similar argument of \cite[Theorem~6.2]{forman2}. 
 \end{remark}
 
 The next corollary follows from Lemma~\ref{le-9.1}.  
 
 \begin{corollary}\label{co-rev-543}
  Let $\mathbb{V}$ be a proper discrete gradient vector field on $\mathcal{H}'$. Then $\mathbb{V}$ is also a proper discrete gradient vector field on $\mathcal{H}$. 
 \end{corollary}
 
 \begin{proof}
 Let $\mathbb{V}$ be a proper discrete gradient vector field on $\mathcal{H}'$.  By Lemma~\ref{le-9.1},  $\mathbb{V}$  is a discrete gradient vector field on $\mathcal{H}$.  Since $\mathbb{V}$  is proper on $\mathcal{H}'$,  each hyperedge of $\mathcal{H}'$ appears in at most one pair  in $\mathbb{V}$.      Let $\alpha\in \mathcal{H}$.  Then we  have the following cases:
 
 {\sc Case~1}:  $\alpha\notin \mathcal{H}'$. 
 
 Then   $\alpha$  does not appear  in any  pair  in $\mathbb{V}$.  
 
  {\sc Case~2}:  $\alpha\in \mathcal{H}'$.

 Then $\alpha$  appears in at most one pair  in $\mathbb{V}$. 
 
Summarizing both cases,  it follows that  each hyperedge of $\mathcal{H}$   appears in at most one pair  in $\mathbb{V}$.   Hence $\mathbb{V}$  is  proper   on $\mathcal{H}$  as well. 
 \end{proof}

The next corollary is a consequence of Lemma~\ref{le-9.1}.

\begin{corollary}\label{co-9.1}
Let $\mathbb{V}$ be a proper discrete gradient vector field on $\mathcal{H}$. Then $\mathbb{V}$ extends to   a proper discrete gradient vector field $\overline {\mathbb{V}}$ on $\Delta\mathcal{H}$ such that $\overline {\mathbb{V}}\mid _\mathcal{H}=\mathbb{V}$ and  $\overline {\mathbb{V}}\mid_{\Delta\mathcal{H}\setminus\mathcal{H}}=0$.  
\end{corollary}

\begin{proof}
We  substitute the pair $\mathcal{H}'\subseteq \mathcal{H}$ of hypergraphs in Corollary~\ref{co-rev-543}  with the pair $\mathcal{H}\subseteq \Delta\mathcal{H}$.  Then from the proper discrete gradient vector field $\mathbb{V}$ on $\mathcal{H}$  we obtain  a proper discrete gradient vector field $\overline {\mathbb{V}}$ on $\Delta\mathcal{H}$ such that $\overline {\mathbb{V}}\mid _\mathcal{H}=\mathbb{V}$ and  $\overline {\mathbb{V}}\mid_{\Delta\mathcal{H}\setminus\mathcal{H}}=0$.   
\end{proof}

\begin{remark}
In  Corollary~\ref{co-9.1},  we  use  the following  notations
\begin{itemize}
\item
  $\overline {\mathbb{V}}\mid _\mathcal{H}=\mathbb{V}$  for $R(\overline {\mathbb{V}})\mid _{R_*(\mathcal{H})}=R(\mathbb{V})$;
  \item
  $\overline {\mathbb{V}}\mid_{\Delta\mathcal{H}\setminus\mathcal{H}}=0$  for $R(\overline {\mathbb{V}})\mid_{R_*(\Delta\mathcal{H}\setminus\mathcal{H})}=0$. 
  \end{itemize}
\end{remark}

\smallskip

\section{Discrete Morse Functions and Their  Gradients on Hypergraphs
 }\label{s5}

 Let $\mathcal{H}$  be a hypergraph.  Let $\Delta\mathcal{H}$  be the associated simplicial complex and $\delta\mathcal{H}$  be the lower-associated simplicial complex.  Let $\overline f$  be a discrete Morse function on $\Delta\mathcal{H}$.   Let $f$  be the restriction of $\overline f$  on $\mathcal{H}$  and $\underline f$  be the restriction  of $\overline f$  on $\delta\mathcal{H}$.  Then $f$  is a discrete Morse function on $\mathcal{H}$  and $\underline f$  is a discrete Morse function on $\delta\mathcal{H}$.   In this section, we study the discrete gradient vector field of $\overline f$  on $\Delta\mathcal{H}$,   the discrete gradient vector field of $  f$  on $ \mathcal{H}$,   and  the discrete gradient vector field of $\underline f$  on $\delta\mathcal{H}$. 
 
 \smallskip

\subsection{Restrictions and Extensions of The  Discrete Gradient Vector Fields} 

Let $\mathcal{H}$  and $\mathcal{H}'$  be hypergraphs  such  that  $\mathcal{H}'\subseteq\mathcal{H}$.  Let $f$ and $f'$  be discrete Morse functions on $\mathcal{H}$ and $\mathcal{H}'$ respectively such that $f'=f\mid_{\mathcal{H}'}$.  We consider the  discrete gradient vector fields  $\grad  f'$  on $\mathcal{H}'$  and  $\grad  f$  on $\mathcal{H}$.   

\begin{lemma}\label{le-4.x}
Let $\pi(\mathcal{H},\mathcal{H}')$ be  the canonical projection from $R(\mathcal{H})_*$ to $R(\mathcal{H}')_*$ sending an  $R$-linear combination of hyperedges
\begin{eqnarray*}
\sum_{\sigma_i\in\mathcal{H}', x_i\in R} x_i\sigma_i + \sum _{\tau_j\in\mathcal{H}\setminus \mathcal{H}', y_j\in R}y_j \tau_j
\end{eqnarray*}
to the $R$-linear combination of hyperedges
\begin{eqnarray*}
\sum _{\sigma_i\in \mathcal{H}', x_i\in R} x_i\sigma_i.
\end{eqnarray*}
Then 
\begin{eqnarray}\label{eq-4.81}
R(\grad  f')=\pi(\mathcal{H},\mathcal{H}')\circ R(\grad  f). 
\end{eqnarray}
\end{lemma}

\begin{proof}
Let $\alpha^{(n)}, \beta^{(n+1)}\in \mathcal{H}$.  We  divide the proof into the following  steps.  

\smallskip

{\bf  Step~1}.  Suppose $\alpha,\beta\in\mathcal{H}' $.   Then 
\begin{eqnarray*}
R(\grad  f')(\alpha)=-\langle\partial\beta,\alpha\rangle\beta
\end{eqnarray*}
if  and only if both of the followings are satisfied:
\begin{itemize}
\item
$\alpha < \beta$; 
\item
$  f'(\beta)\leq f'(\alpha)$
\end{itemize}
if  and only if both of the followings are satisfied:
\begin{itemize}
\item
$\alpha < \beta$; 
\item
$  f(\beta)\leq f(\alpha)$
\end{itemize}
if  and only  if 
\begin{eqnarray*}
R(\grad  f)(\alpha)=-\langle\partial\beta,\alpha\rangle\beta.  
\end{eqnarray*}

\smallskip

{\bf  Step~2}.  Suppose $\alpha \in\mathcal{H}' $.   Then 
\begin{eqnarray*}
R(\grad  f')(\alpha)=0 
\end{eqnarray*}
if  and  only  if  there does not  exist any  $\beta\in\mathcal{H}'$   such that  both of  the followings are satisfied:
\begin{itemize}
\item
$\alpha<\beta$;
\item
$f'(\beta)\leq f'(\alpha)$
\end{itemize}
if  and  only  if  there does not  exist any  $\beta\in\mathcal{H}'$   such that  both of  the followings are satisfied:
\begin{itemize}
\item
$\alpha<\beta$;
\item
$f(\beta)\leq f(\alpha)$
\end{itemize}
if   and  only  if  either  of the followings  is   satisfied:
\begin{itemize}
\item
there does not  exist any  $\gamma\in\mathcal{H}$   such that  both of  the followings are satisfied:
\begin{itemize}
\item
$\alpha<\gamma$;
\item
$f(\gamma)\leq f(\alpha)$
\end{itemize}
\item
there exists    $\gamma\in\mathcal{H}\setminus \mathcal{H}'$   such that  both of  the followings are satisfied:
\begin{itemize}
\item
$\alpha<\gamma$;
\item
$f(\gamma)\leq f(\alpha)$
\end{itemize}
\end{itemize}
if   and  only  if  either  of the followings   is   satisfied:
\begin{itemize}
\item
$R(\grad  f)(\alpha)=0$; 
\item
there exists  $\gamma\in\mathcal{H}\setminus \mathcal{H}'$   such that
$
R(\grad  f)(\alpha)=-\langle\partial\gamma,\alpha\rangle\gamma    
$.  
\end{itemize}

\smallskip

{\bf  Step~3}.  Suppose $\alpha \in\mathcal{H}' $.    Then   either of the followings  is  satisfied:  
\begin{itemize}
\item
$R(\grad  f')(\alpha)=0$; 
\item
$R(\grad  f')(\alpha)=-\langle\partial\beta,\alpha\rangle\beta$  for  some $\beta\in\mathcal{H}'$.   
\end{itemize}

\smallskip

{\bf Step~4}.   Suppose $\alpha \in\mathcal{H}' $.  
We  consider the following  cases:

{\sc Case~1}.  $R(\grad  f')(\alpha)=0$.

{\sc Subcase~1.1}.  $R(\grad  f)(\alpha)=0$.  

Then $R(\grad  f')(\alpha)=\pi(\mathcal{H},\mathcal{H}')\circ R(\grad  f)(\alpha)=0$.

{\sc Subcase~1.2}. $R(\grad  f)(\alpha)\neq 0$.

Then   by  the last line  in  Step~2,  we  have that 
\begin{eqnarray*}
R(\grad  f)(\alpha)=-\langle\partial\gamma,\alpha\rangle\gamma
\end{eqnarray*}
  for  some  $\gamma\in\mathcal{H}\setminus \mathcal{H}'$.   Since 
  \begin{eqnarray*}
  \pi(\mathcal{H},\mathcal{H}')(-\langle\partial\gamma,\alpha\rangle\gamma)=0, 
  \end{eqnarray*}
it follows that 
\begin{eqnarray*}
\pi(\mathcal{H},\mathcal{H}')\circ R(\grad  f)(\alpha)=0=R(\grad  f')(\alpha).   
\end{eqnarray*}

{\sc Case~2}.  $R(\grad  f')(\alpha)\neq  0$.

Then by  Step~3,  we  have that  $R(\grad  f')(\alpha)=-\langle\partial\beta,\alpha\rangle\beta$  for  some $\beta\in\mathcal{H}'$.  By  Step~1,  we  have  
\begin{eqnarray*}
\pi(\mathcal{H},\mathcal{H}')\circ R(\grad  f)(\alpha)&=& \pi(\mathcal{H},\mathcal{H}')(-\langle\partial\beta,\alpha\rangle\beta)\\
&=&-\langle\partial\beta,\alpha\rangle\beta\\
&=&R(\grad  f')(\alpha).  
\end{eqnarray*}
Summarizing  the cases,  it follows that  for any $\alpha\in\mathcal{H}'$   we   always  have   
\begin{eqnarray*}
R(\grad  f')(\alpha)=\pi(\mathcal{H},\mathcal{H}')\circ R(\grad  f)(\alpha).  
\end{eqnarray*}

\smallskip

{\bf  Step~5}.  By the last line in Step~4 as well as the linear property of $R(\grad  f)$,  $R(\grad  f')$  and  $\pi(\mathcal{H},\mathcal{H}')$,   we  obtain (\ref{eq-4.81})  and finish the proof.  
\end{proof}

Let $\mathcal{H}$  be  a hypergraph.   The next proposition follows with the help of  Corollary~\ref{co-9.1}.

\begin{proposition}\label{pr-9.1}
Suppose $\mathcal{H}$ is a hypergraph satisfying the condition (C) (cf.  Proposition~\ref{le-2.a}). 
Let $g$ be a discrete Morse function on $\mathcal{H}$.  Then there exists a discrete Morse function $\overline f$ on $\Delta\mathcal{H}$ such that both of the followings are satisfied: 
\begin{enumerate}[(i).]
\item
$\grad   f=\grad   g$ where $f=\overline f\mid_\mathcal{H}$; 
\item
 $(\grad   \overline f)\mid_{\Delta\mathcal{H}\setminus\mathcal{H}}=0$. 
\end{enumerate}
\end{proposition}

\begin{proof}
Suppose $\mathcal{H}$ is a hypergraph satisfying the  condition (C).   Let $g$  be  a discrete Morse function on $\mathcal{H}$.  Then  by Definition~\ref{def2},  Definition~\ref{def222}  and Proposition~\ref{le-2.a},     we  have that for  any $\alpha\in\mathcal{H}$,  the conditions (A) and (B)  in Lemma~\ref{le-ab}    cannot both be true.   It follows that  $\grad   g$ is a proper discrete gradient vector field on $\mathcal{H}$.  
By Corollary~\ref{co-9.1},  $\grad   g$ extends to a proper discrete gradient vector field $\overline {\grad   g }$ on $\Delta\mathcal{H}$ such that
\begin{eqnarray*}
\overline{\grad   g} \mid_\mathcal{H}=\grad   g
\end{eqnarray*}
and
\begin{eqnarray*}
\overline{\grad   g}\mid_{\Delta\mathcal{H}\setminus\mathcal{H}}=0. 
\end{eqnarray*}
By \cite[Theorem~3.5]{forman2},  there exists a discrete Morse function $\overline f$ on $\Delta\mathcal{H}$ such that 
\begin{eqnarray*}
\overline {\grad   g}=\grad  \overline f.
\end{eqnarray*}
  Let  $f=\overline f\mid _{\mathcal{H}}$  be the restriction of $\overline f$ to $\mathcal{H}$.  Then  
\begin{eqnarray*}
\grad   f = \grad   (\overline f\mid_{\mathcal{H}})
= (\grad  \overline f)\mid_{\mathcal{H}}=\overline {\grad   g}\mid_\mathcal{H}
= {\grad   g}  
\end{eqnarray*}
and
\begin{eqnarray*}
\grad  {\overline f}\mid_{\Delta\mathcal{H}\setminus\mathcal{H}}=\overline{\grad   g}\mid_{\Delta\mathcal{H}\setminus\mathcal{H}}=0. 
\end{eqnarray*}
We  finish the proof. 
\end{proof}

\begin{remark}
By Proposition~\ref{pr-9.1}, in order to study the  discrete gradient vector fields of discrete Morse functions on the hypergraphs satisfying the  condition (C),  it is sufficient to study the discrete gradient vector fields of discrete Morse functions on  the associated simplicial complexes as well as the restrictions of the discrete gradient vector fields to the hypergraphs. 
\end{remark}

\smallskip



\smallskip

\subsection{Discrete Gradient Vector Fields and Critical Hyperedges}

 By Corollary~\ref{cor1},  we let $\overline f$ and ${\underline   f}$ be  the  discrete Morse functions on $\Delta\mathcal{H}$ and $\delta\mathcal{H}$ respectively such that $\underline f=\overline f\mid_{\delta\mathcal{H}}$.  Let $f=\overline f\mid_{\mathcal{H}}$ be the discrete Morse function on $\mathcal{H}$.  We  consider the discrete gradient vector fields 
 \begin{enumerate}[(i).]
 \item
  $\mathbb{V}=\grad   f$  on $\mathcal{H}$, 
  \item
  $\overline{\mathbb{V}}=\grad \overline f$ on $\Delta\mathcal{H}$, 
  \item
  $\underline {\mathbb{V}}= \grad \underline f$ on $\delta\mathcal{H}$. 
  \end{enumerate}

 Let $R(\mathbb{V})$,  $R(\overline{\mathbb{V}})$  and $R(\underline {\mathbb{V}})$  be the  $R$-linear maps   induced by  $\mathbb{V}$,  $\overline{\mathbb{V}}$  and  $\underline{\mathbb{V}}$  respectively.   The next lemma follows.   

\begin{lemma}\label{pr-5}
\begin{enumerate}[(i).]
\item
$R(\mathbb{V})\circ R(\mathbb{V})=0$; 
\item
$M(f,\mathcal{H})=\{\sigma\in\mathcal{H}\mid \sigma\notin {\rm Im}(R(\mathbb{V})) \text{\  and \ } R(\mathbb{V})(\sigma)=0\}$; 
\item
$\#\{\gamma^{(n-1)}\in\mathcal{H}\mid  R(\mathbb{V})(\gamma^{(n-1)})=\pm \alpha \}\leq 1$ for any $\alpha^{(n)}\in \mathcal{H}$; 
\item
The following diagram commutes
\begin{eqnarray*}
\xymatrix{
C_n( \delta\mathcal{H};R)\ar[d]^{R({\underline  { \mathbb{V}}})} \ar[rr]^{{\rm inclusion}}&& R(\mathcal{H})_n\ar[d]^{R(\mathbb{V})} \ar[rr]^{{\rm inclusion}} && C_n( \Delta\mathcal{H};R)\ar[d]^{R(\overline {\mathbb{V}})} \\
C_{n+1}(\delta\mathcal{H};R)&& R(\mathcal{H})_{n+1}\ar[ll]_{\pi(\mathcal{H},\delta\mathcal{H})} && C_{n+1}(\Delta\mathcal{H};R)\ar[ll]_{\pi(\Delta\mathcal{H},\mathcal{H})}. 
}
\end{eqnarray*}
\end{enumerate}
\end{lemma}

\begin{proof}
The proofs  of  (i), (ii), (iii) are similar with the proofs of (1), (2), (3) of \cite[Theorem~6.3]{forman1} respectively.  The proof of (iv) follows from Lemma~\ref{le-4.x}. 
\end{proof}

With the help of Proposition~\ref{pr-1.88} and Lemma~\ref{pr-5},  the next theorem follows. 
  
  \begin{theorem}\label{co-4.a}
  Let $\sigma\in \mathcal{H}$. Then 
 \begin{eqnarray}\label{eq-4.s} 
\sigma\in M(f, \mathcal{H})\setminus (M(\overline f,\Delta\mathcal{H})\cap \mathcal{H})
\end{eqnarray}
 if  and only if   one of the followings holds
  \begin{enumerate}[(i).]
  \item
 $R(\overline {\mathbb{V}})(\sigma)=\pm \eta$ for some $\eta \in \Delta\mathcal{H}\setminus \mathcal{H}$ and  $R(\overline {\mathbb{V}})(\tau)\neq\pm\sigma$  for any $\tau\in \Delta\mathcal{H}$; 
   \item
$R(\overline {\mathbb{V}})(\sigma)=\pm\eta$ for some $\eta\in \Delta\mathcal{H}\setminus \mathcal{H}$   and   there exists $\tau\in \Delta\mathcal{H}\setminus \mathcal{H}$ such that $R(\overline {\mathbb{V}})(\tau)=\pm\sigma$;
   \item
$R(\overline {\mathbb{V}})(\sigma)=0$    and   there exists $\tau\in \Delta\mathcal{H}\setminus \mathcal{H}$ such that $R(\overline {\mathbb{V}})(\tau)=\pm\sigma$. 
  \end{enumerate}
  \end{theorem}
  
  \begin{proof}  We  divide the proof into the following  four  steps.  
  
  \smallskip
  
    {\bf Step~1}. 
      Let $\sigma\in \mathcal{H}$.  We  apply  Lemma~\ref{pr-5}~(ii). 
  \begin{enumerate}[(a).]
  \item
By applying  Lemma~\ref{pr-5}~(ii)  on $\overline f$  and $\Delta\mathcal{H}$,   we have that 
\begin{eqnarray*}
\sigma\in M(\overline f,\Delta\mathcal{H})\cap \mathcal{H}
\end{eqnarray*}
 if  and only if   both  of the followings are satisfied:
 \begin{itemize}
 \item
$
R( \overline {\mathbb{V}})(\sigma)=0;
$
 \item
$
 R(\overline {\mathbb{V}})(\tau)\neq \pm\sigma 
$
 for any $\tau\in \Delta\mathcal{H}$.  
 \end{itemize}
By taking the contrapositive  statement,   we  have that  
   \begin{eqnarray*}
   \sigma\notin M(\overline f,\Delta\mathcal{H})\cap \mathcal{H}
   \end{eqnarray*}
    if   and  only  if   either  of the followings are satisfied: 
    \begin{itemize}
    \item
    $ R(\overline {\mathbb{V}})(\sigma)\neq 0$; 
    \item
    there exists $\tau\in \Delta\mathcal{H}$ such that $R(\overline {\mathbb{V}})(\tau)= \pm\sigma$.     
    \end{itemize}
\item
By  applying  Lemma~\ref{pr-5}~(ii)  to $f$  and  $\mathcal{H}$,   we  have that 
 \begin{eqnarray*}
 \sigma\in M(f,\mathcal{H})
 \end{eqnarray*}
   if  and  only  if  both  of the followings  are  satisfied:
   \begin{itemize}
   \item
    $R(\mathbb{V}) (\sigma)=0$; 
  \item
 $
 R( \mathbb{V})(\tau)\neq\pm\sigma 
  $
  for any $\tau\in \mathcal{H}$. 
   \end{itemize}
\end{enumerate}

\smallskip

{\bf Step~2}.   We  apply  Lemma~\ref{pr-5}~(iv).  

\begin{enumerate}[(a).]
\item
For any $\sigma\in\mathcal{H}$  we  have that 
\begin{eqnarray*}
R(\mathbb{V})(\sigma)=0
\end{eqnarray*}
if  and  only  if 
\begin{eqnarray*}
\pi(\Delta\mathcal{H},\mathcal{H}) \circ  R(\overline{\mathbb{ V}})(\sigma)=0
\end{eqnarray*}
if  and only  if either of the followings  are satisfied:
\begin{itemize}
\item
$R(\overline {\mathbb{V}})(\sigma)=0$;
\item
$R( \overline {\mathbb{V}})(\sigma)=\pm\eta \text{ for some } \eta\in \Delta\mathcal{H}\setminus \mathcal{H}$.  
\end{itemize}

\item
 For any $\sigma,\tau\in\mathcal{H}$  we  have that 
  \begin{eqnarray*}
  R( \mathbb{V})(\tau)\neq\pm\sigma  
   \end{eqnarray*}
   if and  only  if 
  \begin{eqnarray*}
  \pi(\Delta\mathcal{H},\mathcal{H})\circ R(\overline {\mathbb{V}})(\tau)\neq\pm\sigma 
  \end{eqnarray*}
  if  and only  if 
  \begin{eqnarray*}
  R( \overline {\mathbb{V}})(\tau)\neq\pm\sigma
   \end{eqnarray*}
   if  and  only if  either of the followings are satisfied:
   \begin{itemize}
   \item
   $R(\overline {\mathbb{V}})(\tau)=0$; 
   \item
   $R( \overline {\mathbb{V}})(\tau)=\pm \eta \text{ for some }\eta\in \Delta\mathcal{H}\setminus\mathcal{H}$. 
   \end{itemize}

\end{enumerate}
  \smallskip
  
  {\bf  Step~3}.     Summarizing  Step~1  and  Step~2,  for any $\sigma\in \mathcal{H}$  we  have  that  (\ref{eq-4.s}) holds   if  and  only if 
  both of the followings are satisfied:
  \begin{itemize}
  \item 
  $\sigma\in M(f, \mathcal{H})$;   
  \item
  $\sigma\notin (M(\overline f,\Delta\mathcal{H})\cap \mathcal{H})$
  \end{itemize}
 which happens  if  and only if  
   both of  the followings are satisfied:
  
  \begin{itemize}
  \item
  both of the followings are satisfied: 
  \begin{itemize}
  \item
  either of the followings  are satisfied:
\begin{itemize}
\item
$R(\overline {\mathbb{V}})(\sigma)=0$;    
\item
$R( \overline {\mathbb{V}})(\sigma)=\pm\eta \text{ for some } \eta\in \Delta\mathcal{H}\setminus \mathcal{H}$;   
\end{itemize}

  \item
    $R(\overline {\mathbb{V}})(\tau)\neq\pm\sigma$  for any $\tau\in\mathcal{H}$;     
  
  \end{itemize}
 \item
  either  of the followings are satisfied: 
    \begin{itemize}
    \item
    $ R(\overline {\mathbb{V}})(\sigma)\neq 0$; 
    \item
    there exists $\tau\in \Delta\mathcal{H}$ such that $R(\overline {\mathbb{V}})(\tau)= \pm\sigma$.     
    \end{itemize}

 \end{itemize}

\smallskip

{\bf  Step~4}.  Let  $\sigma\in M(f, \mathcal{H})\setminus (M(\overline f,\Delta\mathcal{H})\cap \mathcal{H})
$.  We  consider the following cases.  

  {\sc Case~1}.  $R(\overline {\mathbb{V}})(\sigma)\neq  0$.

  Then  by   line 9 in Step~3,   we  have that  $R(\overline {\mathbb{V}})(\sigma)=\pm \eta$ for some $\eta \in \Delta\mathcal{H}\setminus \mathcal{H}$.  By  the last three  lines  in Step~3,  we have the following subcases.

   {\sc Subcase~1.1}.  For any $\tau\in\Delta\mathcal{H}$,  we  always  have  $R(\overline {\mathbb{V}})(\tau)\neq\pm\sigma$.  
   
   Then we obtain Theorem~\ref{co-4.a}~(i).

   {\sc Subcase~1.2}.  There exists $\tau\in \Delta\mathcal{H}$ such that  $R(\overline {\mathbb{V}})(\tau)=\pm\sigma$.

   Then by  line 10  of  Step~3,  we   have $\tau\in \Delta\mathcal{H}\setminus \mathcal{H}$.  Thus  we obtain  Theorem~\ref{co-4.a}~(ii). 
   
   {\sc Case~2}.  $R(\overline {\mathbb{V}})(\sigma)=  0$.

     Then  by  the last line  in  Step~3,  there exists $\tau\in \Delta\mathcal{H}$ such that $R(\overline {\mathbb{V}})(\tau)= \pm\sigma$.     By line 10  in  Step~3, 
      we have $\tau\in \Delta\mathcal{H}\setminus\mathcal{H}$.  We obtain  Theorem~\ref{co-4.a}~(iii).

   Summarizing  the above cases,   we finish the proof.  
     \end{proof}

  The next corollary is a consequence of Theorem~\ref{co-4.a}.

  \begin{corollary}\label{co-4.aa}
 Suppose $R(\overline {\mathbb{V}})(\alpha)= R(\mathbb{V})(\alpha)$ for $\alpha\in\mathcal{H}$ and $R(\overline {\mathbb{V}})(\alpha)=0$ for $\alpha\in\Delta\mathcal{H}\setminus \mathcal{H}$.  Then
 \begin{eqnarray*}
 M(f,\mathcal{H})= M(\bar f,\Delta\mathcal{H})\cap\mathcal{H}. 
 \end{eqnarray*} 
  \end{corollary}
  
  \begin{proof}
  It suffices to prove that  there does not exist any $\sigma\in\mathcal{H}$  such that (\ref{eq-4.s})  holds.  Since we assume $R(\overline {\mathbb{V}})(\alpha)= R(\mathbb{V})(\alpha)$ for $\alpha\in\mathcal{H}$ and $R(\overline {\mathbb{V}})(\alpha)=0$ for $\alpha\in\Delta\mathcal{H}\setminus \mathcal{H}$,  it is direct that there does not exist any $\sigma\in\mathcal{H}$  such that at least one of (i), (ii), or (iii) in Theorem~\ref{co-4.a}  holds.  Consequently,  by Theorem~\ref{co-4.a},   there does not exist any $\sigma\in\mathcal{H}$  satisfying (\ref{eq-4.s}). 
  \end{proof}
  
  The next corollary follows from Proposition~\ref{pr-9.1} and Corollary~\ref{co-4.aa}. 
  
  \begin{corollary}\label{co-4.bb}
Suppose $\mathcal{H}$  satisfies the  condition (C) (cf.  Proposition~\ref{le-2.a}). 
Let $g$ be a discrete Morse function on $\mathcal{H}$, $\overline f$ a discrete Morse function on $\Delta\mathcal{H}$ given in Proposition~\ref{pr-9.1}, and $f=\overline f\mid_{\mathcal{H}}$.  Then 
\begin{eqnarray*}
M(g,\mathcal{H})= M(f,\mathcal{H})= M(\bar f,\Delta\mathcal{H})\cap\mathcal{H}. 
\end{eqnarray*}
  \end{corollary}
  
  \begin{proof}
  By Proposition~\ref{pr-9.1}~(i),  we  have $\grad   f=\grad   g$.  Thus $M(g,\mathcal{H})= M(f,\mathcal{H})$.  By Proposition~\ref{pr-9.1}~(i) and (ii),  we  have 
  \begin{eqnarray*}
  R(\overline {\grad   f})(\alpha)= R(\grad   f)(\alpha)
  \end{eqnarray*}
   for $\alpha\in\mathcal{H}$ and  
   \begin{eqnarray*}
   R(\overline {\grad   f})(\alpha)=0
   \end{eqnarray*}
    for $\alpha\in\Delta\mathcal{H}\setminus \mathcal{H}$.  Thus by Corollary~\ref{co-4.aa},    we have  $M(f,\mathcal{H})= M(\bar f,\Delta\mathcal{H})\cap\mathcal{H}$.  We finish the proof. 
  \end{proof}
  
  \smallskip

\section{An Example}\label{s6}

\begin{example}[cf.  Figure~2]\label{ex-6.8}
Consider a hypergraph
\begin{eqnarray*}
 \mathcal{H}=\{\{v_0\},\{v_1\},\{v_2\},\{v_3\},\{v_0,v_1\},\{v_0,v_3\}, \{v_1,v_3\}, \{v_0,v_1,v_2\}\}.  
\end{eqnarray*}
\begin{enumerate}[(i).]
\item
The infimum chain complex is 
\begin{eqnarray*}
{\rm Inf}_0(\mathcal{H})&=&R(\{v_0\},\{v_1\},\{v_2\},\{v_3\}),\\
 {\rm Inf}_1(\mathcal{H})&=&R(\{v_0,v_1\},\{v_0,v_3\}, \{v_1,v_3\}),\\
 {\rm Inf}_2(\mathcal{H})&=&0
\end{eqnarray*}
and the supremum chain complex is 
\begin{eqnarray*}
{\rm Sup}_0(\mathcal{H})&=&R(\{v_0\},\{v_1\},\{v_2\},\{v_3\}),\\
{\rm Sup}_1(\mathcal{H})&=&R(\{v_0,v_1\},\{v_0,v_3\}, \{v_1,v_3\},\{v_1,v_2\}-\{v_0,v_2\}), \\
 {\rm Sup}_2(\mathcal{H})&=&R(\{v_0,v_1,v_2\}).  
\end{eqnarray*}
\item
  By a direct calculation,    the embedded homology is 
\begin{eqnarray*}
H_0(\mathcal{H};R)={R}^{\oplus 2}, \text{\ \ \ } H_1(\mathcal{H};R)={R}, \text{\ \ \ } H_2(\mathcal{H};R)=0. 
\end{eqnarray*}

\item
 The associated simplicial complex $\Delta\mathcal{H}$ is the simplicial complex obtained by adding the $1$-simplices
\begin{eqnarray*}
\{v_0,v_2\},   \{v_1,v_2\}
\end{eqnarray*}
to $\mathcal{H}$.

\item
The  lower-associated simplicial complex $\delta\mathcal{H}$ is the   discrete  set of  the  vertices $\{v_0,v_1,v_2,v_3\}$.  

\item 
Let $\overline f$  be a discrete Morse function on $\Delta\mathcal{H}$  given by
\begin{eqnarray*}
&\overline f(\{v_0\})=1,\\
& \overline f(\{v_1\})=\overline f(\{v_2\})=\overline f(\{v_3\})=0,\\
& \overline f (\{v_0,v_1\})=\overline f(\{v_1,v_2\})=\overline  f(\{v_1,v_3\})=1,\\
& \overline f(\{v_0,v_2\})=\overline f(\{v_0,v_3\})=2,\\
&\overline f (\{v_0,v_1,v_2\})=2. 
\end{eqnarray*}
Then  the set of the critical simplices of $\overline{f}$  in $\Delta\mathcal{H}$ is  
\begin{eqnarray*}
M(\overline f,  \Delta\mathcal{H})=\{\{v_1\},\{v_2\},\{v_3\}, \{v_0,v_3\}, \{v_1,v_2\}, \{v_1,v_3\}\}.  
\end{eqnarray*}
Let $\overline {\mathbb{V}}= \grad   \overline f$ be the discrete gradient vector field on $\Delta\mathcal{H}$. Then
\begin{eqnarray*}
&\overline {\mathbb{V}}(\{v_0\})=\{v_0,v_1\},\\
&\overline {\mathbb{V}}(\{v_0,v_2\})=\{v_0,v_1,v_2\},\\
&\overline {\mathbb{V}}(\sigma)=0 \text{ for any }\sigma\in \Delta\mathcal{H}\setminus\{\{v_0\}, \{v_0,v_2\}\}.
\end{eqnarray*}

\item
Let $f$  be the restriction of $\overline f$  on  $\mathcal{H}$.  Then  $f$  is given by
\begin{eqnarray*}
& f(\{v_0\})=1,\\
& f(\{v_1\})=f(\{v_2\})= f(\{v_3\})=0,\\
& f(\{v_0,v_1\})=  f(\{v_1,v_3\})=1,\\
& f(\{v_0,v_3\})=2,\\
&f (\{v_0,v_1,v_2\})=2. 
\end{eqnarray*}
Moreover,  the  set of the critical hyperedges  of $f$  in  $\mathcal{H}$  is 
 \begin{eqnarray*}
M(f,   \mathcal{H})=\{\{v_1\},\{v_2\},\{v_3\}, \{v_0,v_3\},  \{v_1,v_3\}, \{v_0,v_1,v_2\}\}.  
\end{eqnarray*}
Let $  {\mathbb{V}}=\grad     f$ be the discrete gradient vector field on $ \mathcal{H}$. Then
\begin{eqnarray*}
& {\mathbb{V}}(\{v_0\})=\{v_0,v_1\},\\
&  {\mathbb{V}}(\{v_0,v_2\})=\{v_0,v_1,v_2\},\\
&  {\mathbb{V}}(\sigma)=0 \text{ for any }\sigma\in   \mathcal{H}\setminus\{\{v_0\}, \{v_0,v_2\}\}.
\end{eqnarray*}

\item
Let $\underline{f}$  be the restriction of $f$  on  $\delta\mathcal{H}$.  Then  $\underline f$  is given by
\begin{eqnarray*}
  f(\{v_0\})=1, ~~~~~~
  f(\{v_1\})=f(\{v_2\})= f(\{v_3\})=0.
\end{eqnarray*}
Moreover,  the  set of the critical  simplices  of $\underline f$  in  $\delta\mathcal{H}$  is 
 \begin{eqnarray*}
M(\underline {f},   \delta\mathcal{H})=\{\{v_0\},\{v_1\},\{v_2\},\{v_3\}\}.  
\end{eqnarray*}
Let $  \underline {\mathbb{V}}=\grad   {\underline  f}$ be the discrete gradient vector field on $ \delta\mathcal{H}$. Then
\begin{eqnarray*}
  {\mathbb{V}}(\{v_i\})=0, ~~~~~~ i=0,1,2,3.  
   \end{eqnarray*}

\end{enumerate}
\end{example}

\begin{figure}[!htbp]
 \begin{center}
\begin{tikzpicture}[line width=1.5pt]

\coordinate [label=right:$v_0$]    (A) at (1.5,2); 
 \coordinate [label=left:$v_1$]   (B) at (0.9,0); 
 \coordinate  [label=right:$v_2$]   (C) at (2.4,0); 
\coordinate  [label=left:$v_3$]   (D) at (0,2);

\fill (1.5,2) circle (2.5pt);
\fill (0.9,0) circle (2.5pt);
\fill (2.4,0) circle (2.5pt);
\fill (0,2) circle (2.5pt);

 \coordinate[label=left:$\mathcal{H}$:] (G) at (-0.5,1);
 \draw [dashed,thick] (C) -- (B);
 \draw [dashed,thick] (C) -- (A);
  \draw [thick] (B) -- (D);
    \draw [thick] (B) -- (A);
    \draw [thick] (D) -- (A);

\fill [fill opacity=0.25][gray!100!white] (B) -- (A) -- (C) -- cycle;

 


\coordinate [label=right:$v_0$]    (P) at (7.5,2); 
 \coordinate [label=left:$v_1$]   (Q) at (6.9,0); 
 \coordinate  [label=right:$v_2$]   (R) at (8.4,0); 
\coordinate  [label=left:$v_3$]   (S) at (6,2);

\draw[->] (7.5,2) -- (7.35,1.5);

\draw[->] (7.95,1) -- (7.95-1/3,0.9);

\fill (7.5,2) circle (2.5pt);
\fill (6.9,0) circle (2.5pt);
\fill (8.4,0) circle (2.5pt);
\fill (6,2) circle (2.5pt);

 \coordinate[label=left:$\Delta\mathcal{H}$ and $\overline V$:] (M) at (5.5,1);
 \draw [thick] (R) -- (Q);
 \draw [thick] (R) -- (P);
  \draw [thick] (Q) -- (S);
    \draw [thick] (P) -- (Q);
    \draw [thick] (S) -- (P);

\fill [fill opacity=0.25][gray!100!white] (R) -- (P) -- (Q) -- cycle;

 \end{tikzpicture}
\end{center}

\caption{Example~\ref{ex-6.8}.}
\end{figure}

\bigskip

\section*{Acknowledgement}

{The  present authors  would like to express their  deep
gratitude to the referee for the careful reading of the manuscript\footnote[4]{This work was supported by the Singapore Ministry of Education Research Grant (AcRF Tier 1 WB-
S No.R-146-000-222-112),  the  President's Graduate Fellowship of National University of Singapore, the Natural Science Foundation of China (Nos.11971144, 12001310)  and  High-Level Scientific Research Foundation of Hebei Province.}. }

\bigskip

\bigskip
 
 {\small
 Shiquan Ren  (for correspondence) 
 
 Address:  
 
 School of Mathematics and Statistics,  Henan University,  Kaifeng  475004,  China.

  e-mail: srenmath@126.com
  
  \bigskip
  
  Chong Wang
  
  Address: School of Mathematics, Renmin University of China, Beijing 100872,  China;  School of Mathematics and Statistics,  Cangzhou  Normal  University,  Cangzhou    061001,  China.

  e-mail:  wangchong\_618@163.com 
  
  \bigskip

Chengyuan Wu
  
Address: Department of Mathematics, National University of Singapore, Singapore
119076, Singapore; Institute of High Performance Computing, A*STAR, Singapore 138632, Singapore.
  
  e-mail: wuchengyuan@u.nus.edu
  
  \bigskip
  
Jie Wu
  
 School of Mathematical Sciences, Hebei Normal University, Shijiazhuang 050024, China;  Center for
Topology and Geometry based Technology, Hebei Normal University, Shijiazhuang 050024, China.

  
  e-mail: wujie@hebtu.edu.cn
  }

\bigskip


\begin{thebibliography}{99}

 

  







\bibitem{a1}
R.  Ayala, L.M.  Fern\'{a}ndez  and J.A.  Vilches, \emph{Discrete Morse inequalities on infinite graphs}. Electron.  J.  Combin.  {\bf 16} (1) (2009),  R38.  


\bibitem{a2}
R.  Ayala, L.M.  Fern\'{a}ndez  and J.A.  Vilches, \emph{Morse inequalities on certain  infinite $2$-complexes}.   Glasg. Math. J. {\bf 49} (2) (2007),  155-165.  



\bibitem{a3}
R.  Ayala, L.M.  Fern\'{a}ndez,   D. Fern\'{a}ndez-Ternero  and J.A.  Vilches,  \emph{Discrete Morse theory on graphs}.   Topol.  Appl.  {\bf 156} (2) (2009),  3091-3100.   

\bibitem{a4}
R.  Ayala, L.M.  Fern\'{a}ndez,   A.  Quintero  and J.A.  Vilches, \emph{A  note on the pure Morse complex  of a graph}. Topol.  Appl.  {\bf 155}   (2008),  2084-2089.

\bibitem{berge}
 C. Berge,  \emph{Graphs and hypergraphs}. North-Holland Mathematical Library, Amsterdam, 1973.  

\bibitem{h1}
S. Bressan, J. Li, S. Ren and J. Wu, \emph{The embedded homology of hypergraphs and applications}. Asian J. Math. {\bf 23} (3)  (2019), 479-500. 





 





 





 

 

 








\bibitem{forman9}
R. Forman, \emph{A discrete Morse theory for cell complexes}, in Geometry, Topology and Physics in Honor of Raoul Bott, Shing-Tung Yau ed. International Press, 1995. 

\bibitem{formanz}
R. Forman, \emph{Combinatorial vector fields and dynamical systems}. Math. Z.   {\bf 228} (4) (1998),  629-681. 

\bibitem{forman1}
R. Forman, \emph{Morse theory for cell complexes}. Adv. Math. {\bf 134} (1)  (1998), 90-145. 

\bibitem{forman8}
R. Forman, \emph{Witten-Morse theory for cell complexes}. Topology {\bf 37} (5)
 (1998), 945-979.  

\bibitem{forman5}
R. Forman, \emph{ Combinatorial differential topology and geometry}.   New Perspectives in Geometric Combinatorics,  MSRI Publications {\bf 38}, 1999.   

\bibitem{forman4}
R. Forman,   \emph{Morse theory and evasiveness}.  Combinatorica {\bf 20} (4)
 (2000), 489-504. 

\bibitem{forman2}
R. Forman, \emph{A user's guide to discrete Morse theory}. S\'{e}minaire Lotharingien de Combinatoire  {\bf 48} Article B48c, 2002. 

\bibitem{forman3}
R. Forman,  \emph{Discrete Morse theory and the cohomology ring}. 
 Trans.       Amer.  Math.  Soc.  {\bf 354}  (12)
  (2002),  5063-5085.  














 

\bibitem{hatcher}
A. Hatcher, \emph{Algebraic topology}. Cambridge University Press, Cambridge, 2001. 


\bibitem{mams}
M. J\"ollenbeck and V. Welker, \emph{Minimal resolutions via algebraic discrete Morse theory}. Memoirs of the American Mathematical Society {\bf 923},  2009. 

\bibitem{alg2}
D. N. Kozlov,  \emph{Discrete Morse theory for free chain complexes}. Comptes Rendus Mathematique,  C. R. Acad. Sci. Paris, Ser. I {\bf 340} (12) (2005), 867-872.

\bibitem{app1}
T.  Lewiner,   H.  Lopes   and  G.  Tavares,   \emph{Applications of Forman's  discrete Morse  theory  to topology  visualization  and mesh compression}.  Transactions on visualization and computer graphics {\bf 10} (5) (2004), 499-508, IEEE.  


\bibitem{lin}
Y.  Lin,  C. Wang  and  S.-T.  Yau,  \emph{Discrete Morse theory on digraphs}.  2021. arXiv: 2102. 10518v1. 

\bibitem{milnor}
J. W. Milnor,  \emph{The geometric realization of a semi-simplicial complex}. Ann. of Math. {\bf 65} (2) (1957), 357-362.

\bibitem{pers}
K.  Mischaikow  and  V.  Nanda, \emph{Morse theory for filtrations and efficient
computation of persistent homology}. Discrete Comput. Geom. {\bf 50} (2) (2013), 330-353. 




  \bibitem{parks}
  A.D. Parks  and S.L. Lipscomb, \emph{Homology and hypergraph acyclicity: a combinatorial invariant for hypergraphs}.   Naval Surface Warfare Center, 1991.  
   











 



  
 
 


 






  




 

  
 
 

 
\bibitem{algm}
E. Sk\"oldberg,  \emph{Morse theory from an algebraic viewpoint}. Trans. Amer. 
Math. Soc.  {\bf 358} (1)
 (2006),  115-129.
 
 \bibitem{app3}
 H.  Kannan,  E.  Saucan,  I.  Roy  and  A. Samal,  \emph{Persistent homology of unweighted complex networks via discrete Morse theory}.  Scientific Reports {\bf 9} (2019), Article 13817. 

\bibitem{cy}
C.  Wu, S.  Ren, J.  Wu and  K.~L.  Xia, \emph{Discrete Morse Theory for Weighted Simplicial Complexes}.  Topol.   Appl.  {\bf 270} (2020), Article 107038.


\bibitem{wu1}
J. Wu, \emph{Simplicial objects and homotopy groups}, in \emph{Braids. Introductory Lectures on Braids, Configurations and their Applications}. World Scientific, Hackensack, 2010,   31-181. 

 

 
 
 


 
 
 
 

 
 
 
\end{thebibliography}
  \end{document}